\documentclass{article}

\usepackage{amssymb,stmaryrd}
\newcommand{\NN}{\mathbb{N}}
\newcommand{\ZZ}{\mathbb{Z}}
\newcommand{\QQ}{\mathbb{Q}}
\newcommand{\RR}{\mathbb{R}}
\newcommand{\liff}{\Leftrightarrow}
\newcommand{\limp}{\Rightarrow}
\renewcommand{\lnot}{\neg\,}
\newcommand{\IHOL}{\mathrm{IHOL}}
\newcommand{\res}{\upharpoonright}
\newcommand{\ser}{\upharpoonleft}
\newcommand{\dom}{\mathrm{dom}}
\newcommand{\rng}{\mathrm{rng}}
\newcommand{\pow}{P}
\newcommand{\lang}{L}
\newcommand{\calU}{\mathcal{U}}
\newcommand{\calV}{\mathcal{V}}
\newcommand{\AC}{\mathrm{AC}}
\newcommand{\GMP}{\mathrm{GMP}}

\newcommand{\BP}{\mathrm{BP}}
\newcommand{\ACBP}{\mathrm{ACBP}}
\newcommand{\Co}{\mathrm{C}_\mathrm{o}}
\newcommand{\Colc}{\mathrm{C}_\mathrm{o}^\mathrm{lc}}
\newcommand{\Coc}{\mathrm{C}_\mathrm{o}^\mathrm{c}}
\newcommand{\llb}{\llbracket}
\newcommand{\rrb}{\rrbracket}
\newcommand{\Ap}{\mathrm{Ap}\,}
\newcommand{\sh}{\mathrm{sh}}
\newcommand{\Sh}{\mathrm{Sh}}
\newcommand{\leT}{\le_\mathrm{T}}
\newcommand{\eqT}{\equiv_\mathrm{T}}
\newcommand{\degT}{\mathrm{deg}_\mathrm{T}}
\newcommand{\lew}{\le_\mathrm{w}}
\newcommand{\les}{\le_\mathrm{s}}
\newcommand{\eqw}{\equiv_\mathrm{w}}
\newcommand{\wequiv}{\,\,{\vdash\!\dashv}\,\,}
\newcommand{\frakM}{\mathfrak{M}}
\newcommand{\frakA}{\mathfrak{A}}
\newcommand{\frakB}{\mathfrak{B}}
\newcommand{\degw}{\mathrm{deg}_\mathrm{w}}
\newcommand{\degs}{\mathrm{deg}_\mathrm{s}}
\newcommand{\ooo}{\mathbf{0}}
\newcommand{\aaa}{\mathbf{a}}
\newcommand{\bbb}{\mathbf{b}}
\newcommand{\ccc}{\mathbf{c}}
\newcommand{\ddd}{\mathbf{d}}
\newcommand{\imp}{\mathrm{imp}}
\newcommand{\DT}{\mathcal{D}_\mathrm{T}}
\newcommand{\Dw}{\mathcal{D}_\mathrm{w}}
\newcommand{\Ds}{\mathcal{D}_\mathrm{s}}
\newcommand{\Omegai}{\Omega_1}
\newcommand{\Xsh}{{\widehat{X}^\sh}}
\newcommand{\NNsh}{\widehat{\NN}^\sh}
\newcommand{\ZZsh}{\widehat{\ZZ}^\sh}
\newcommand{\QQsh}{\widehat{\QQ}^\sh}
\newcommand{\RRsh}{\widehat{\RR}^\sh}
\newcommand{\congQ}{\cong_\QQ}
\newcommand{\abar}{\overline{a}}
\newcommand{\bbar}{\overline{b}}
\newcommand{\cbar}{\overline{c}}
\newcommand{\dbar}{\overline{d}}
\newcommand{\nbar}{\overline{n}}

\newcommand{\Abar}{\overline{A}}
\newcommand{\Bbar}{\overline{B}}
\newcommand{\Cbar}{\overline{C}}
\newcommand{\Pbar}{\overline{P}}
\newcommand{\Qbar}{\overline{Q}}
\newcommand{\Omegabar}{\overline{\Omega}}
\newcommand{\ahat}{\widehat{a}}
\newcommand{\bhat}{\widehat{b}}
\newcommand{\chat}{\widehat{c}}

\newcommand{\xhat}{\widehat{x}}

\usepackage{amsthm}
\theoremstyle{definition}
\newtheorem{thm}{Theorem}[section]
\newtheorem{lem}[thm]{Lemma}
\newtheorem{cor}[thm]{Corollary}
\newtheorem{dfn}[thm]{Definition}
\newtheorem{rem}[thm]{Remark}
\newtheorem{exa}[thm]{Example}
\newtheorem{thmA}{Theorem}
\newtheorem{dfnA}{Definition}
\newtheorem{corA}{Corollary}
\newtheorem*{cor*A}{Corollary}
\newtheorem*{lem*A}{Lemma}

\usepackage[backref=page,colorlinks=true,allcolors=blue]{hyperref}

\begin{document}

\begin{center}
  \LARGE Mass problems and\\
  intuitionistic higher-order logic\\[18pt]
  \large Sankha S. Basu\\
  Department of Mathematics\\
  Pennsylvania State University\\
  University Park, PA 16802, USA\\
  \href{http://www.personal.psu.edu/ssb168}
  {http://www.personal.psu.edu/ssb168}\\
  \href{mailto:basu@math.psu.edu}{basu@math.psu.edu}\\[12pt]
  Stephen G. Simpson\footnote{Simpson's research was partially
    supported by the Eberly College of Science at the Pennsylvania
    State University, and by Simons Foundation
    Collaboration Grant 276282.}\\
  Department of Mathematics\\
  Pennsylvania State University\\
  University Park, PA 16802, USA\\
  \href{http://www.math.psu.edu/simpson}{http://www.math.psu.edu/simpson}\\
  \href{mailto:simpson@math.psu.edu}{simpson@math.psu.edu}\\[12pt]
  First draft: February 7, 2014\\
  This draft: \today\\
  {\ }
\end{center}

\begin{abstract}
  \addcontentsline{toc}{section}{Abstract} In this paper we study a
  model of intuitionistic higher-order logic which we call \emph{the
    Muchnik topos}.  The Muchnik topos may be defined briefly as the
  category of sheaves of sets over the topological space consisting of
  the Turing degrees, where the Turing cones form a base for the
  topology.  We note that our Muchnik topos interpretation of
  intuitionistic mathematics is an extension of the well known
  Kolmogorov/Muchnik interpretation of intuitionistic propositional
  calculus via Muchnik degrees, i.e., mass problems under weak
  reducibility.  We introduce a new sheaf representation of the
  intuitionistic real numbers, \emph{the Muchnik reals}, which are
  different from the Cauchy reals and the Dedekind reals.  Within the
  Muchnik topos we obtain a \emph{choice principle} $(\forall
  x\,\exists y\,A(x,y))\limp\exists w\,\forall x\,A(x,wx)$ and a
  \emph{bounding principle} $(\forall x\,\exists
  y\,A(x,y))\limp\exists z\,\forall x\,\exists y\,(y\leT(x,z)\land
  A(x,y))$ where $x,y,z$ range over Muchnik reals, $w$ ranges over
  functions from Muchnik reals to Muchnik reals, and $A(x,y)$ is a
  formula not containing $w$ or $z$.  For the convenience of the
  reader, we explain all of the essential background material on
  intuitionism, sheaf theory, intuitionistic higher-order logic,
  Turing degrees, mass problems, Muchnik degrees, and Kolmogorov's
  calculus of problems.  We also provide an English translation of
  Muchnik's 1963 paper on Muchnik degrees.
\end{abstract}

\newpage

\tableofcontents

\section{Introduction}
\label{sec:introduction}

\subsection{Intuitionism and the calculus of problems}
\label{subsec:history}

\subsubsection{Constructivism}

In the early part of the 20th century, foundations of mathematics was
dominated by Georg Cantor's set theory and David Hilbert's program of
finitistic reductionism.  The harshest critics of set theory and
Hilbert's program were the \emph{constructivists}.  Among the various
constructivist schools were \emph{intuitionism}, proposed by
L. E. J. Brouwer in 1907; \emph{predicativism}, proposed by Herman
Weyl in 1918; \emph{finitism}, proposed by Thoralf Skolem in 1923;
\emph{constructive recursive mathematics}, proposed by Andrei
Andreyevich Markov in 1950; and \emph{Bishop-style constructivism},
proposed by Errett Bishop in 1967.  Also among the constructivists
were many other prominent mathematicians including Leopold Kronecker
(1823--1891), who is sometimes regarded as ``the first
constructivist,'' Ren\'e Louis Baire (1874--1932), Emile Borel
(1871--1956), Nikolai Nikolaevich Lusin (1883--1950) and Jules Henri
Poincar\'e (1854--1913).  For more about the various schools of
constructivism and their history, see \cite{Troelstra_historyof} and
\cite[Chapter 1]{VDT}.

\subsubsection{Brouwer's intuitionism}

Intuitionism is a constructive approach to mathematics proposed by
Brouwer.  The philosophical basis of intuitionism was spelled out in
Brouwer's 1907 Ph.D.\ thesis, entitled ``On the foundations of
mathematics.''  The mathematical consequences were developed in
Brouwer's subsequent papers, 1912--1928.

The following is quoted from \cite[Chapter 1]{VDT}.
\begin{quote}
  The basic tenets of Brouwer's intuitionism are as follows.
  \begin{enumerate}
  \item Mathematics deals with mental constructions, which are
    immediately grasped by the the mind; mathematics does not consist
    in the formal manipulation of symbols, and the use of mathematical
    language is a secondary phenomenon, induced by our limitations
    (when compared with an ideal mathematician with unlimited memory
    and perfect recall), and the wish to communicate our mathematical
    constructions with others.
  \item It does not make sense to think of truth and falsity of a
    mathematical statement independently of our knowledge concerning
    the statement. A statement is true if we have a proof of it, and
    false if we can show that the assumption that there is a proof for
    the statement leads to a contradiction. For an arbitrary statement
    we can therefore not assert that it is either true or false.
  \item Mathematics is a free creation: it is not a matter of mentally
    reconstructing, or grasping the truth about mathematical objects
    existing independently of us.
  \end{enumerate}
\end{quote}
  
In Brouwer's view, mathematics allows the construction of mathematical
objects on the basis of intuition.  Mathematical objects are mental
constructs, and mathematics is independent of logic\footnote{On the
  contrary, logic is an application or part of mathematics (according
  to Brouwer).} and cannot be founded upon the axiomatic method.  In
particular, Brouwer rejected Hilbert's formalism and Cantor's set
theory.

An important feature of Brouwer's work was \emph{weak
  counterexamples}, introduced to show that certain statements of
classical mathematics are not intuitionistically acceptable. A weak
counterexample to a statement $A$ is not a counterexample in the
strict sense, but rather an argument to the effect that any
intuitionistic proof of $A$ would have to include a solution of a
mathematical problem which is as yet unsolved.

In particular, the \emph{principle of the excluded middle} (PEM),
$A\lor\lnot A$, is valid in classical logic, but to accept it
intuitionistically we would need a universal method for obtaining, for
any $A$, either a proof of $A$ or a proof of $\lnot A$, i.e., a method
for obtaining a contradiction from a hypothetical proof of $A$. But if
such a universal method were available, we would also have a method to
decide the truth or falsity of statements $A$ which have not yet been
proved or refuted (e.g., $A\equiv$ ``there are infinitely many twin
primes''), which is not the case. Thus we have a weak counterexample
to PEM.

The above argument shows that PEM is not intuitionistically
acceptable.  However, intuitionists may accept certain special cases
or consequences of PEM.  In particular, since we cannot hope to find a
proof of $\lnot(A\lor\lnot A)$, it follows that $\lnot\lnot(A\lor\lnot
A)$ is intuitionistically acceptable.

Excessive emphasis on weak counterexamples has sometimes created the
impression that intuitionism is mainly concerned with refutation of
principles of classical mathematics.  However, Brouwer introduced a
number of other innovations, such as \emph{choice sequences}; see
\cite[Chapter 3]{DUMM} and \cite[Chapters 4 and 12]{VDT}.  After 1912
Brouwer developed what has come to be known as \emph{Brouwer's
  program}, which provided an alternative perspective on foundations
of mathematics, parallel to Hilbert's program.  For more on the
history of intuitionism and Brouwer's work, see
\cite{Troelstra_historyof} and \cite[Chapter 1]{VDT}.

\subsubsection{Kolmogorov's calculus of problems}

The great mathematician Andrei Nikolaevich Kolmogorov published two
papers on intuitionism.

In Kolmogorov's 1925 paper \cite{KOL-Russian}\footnote{See also the
  English translation \cite{KOL-PEM}.} he introduces \emph{minimal
  propositional calculus}, which is strictly included in
intuitionistic propositional calculus.  Starting with minimal
propositional calculus, one can add $A\limp(\lnot A\limp B)$ to get
intuitionistic propositional calculus, and then one can add
$(\lnot\lnot A)\limp A$ to get classical propositional calculus.
Furthermore, a propositional formula $A$ is classically provable if
and only if $\lnot\lnot A$ is intuitionistically provable.  This
translation of classical to intuitionistic propositional calculus, due
to Kolmogorov \cite{KOL-Russian}, predates the double-negation
translations of G\"odel and Gentzen.

In Kolmogorov's 1932 paper \cite{KOL-German}\footnote{See also the
  English translation \cite{KOL}.} he gives a natural but non-rigorous
interpretation of intuitionistic propositional calculus, called the
\emph{calculus of problems}.  Each proposition is regarded as a
problem, and logically compound propositions are obtained by combining
simpler problems.  If $A$ and $B$ are problems, then:
\begin{enumerate}
\item $A\land B$ is the problem of solving both problem $A$ and
  problem $B$;
\item $A\lor B$ is the problem of solving either problem $A$ or
  problem $B$;
\item $A\limp B$ is the problem of solving problem $B$ given a
  solution of problem $A$, i.e., of reducing problem $B$ to problem
  $A$; and
\item $\lnot A$ is the problem of showing that problem $A$ has no
  solution;
\end{enumerate}
but Kolmogorov does not give a rigorous definition of ``problem.''
For further discussion of these papers of Kolmogorov, see \cite{COQ}.

Arend Heyting was one of Brouwer's principal students.  His primary
contribution to intuitionism was, ironical as it may sound, the
formalization of intuitionistic logic and arithmetic.  Heyting also
proposed what is now called the \emph{proof interpretation} for
intuitionistic logic.  In this interpretation, the meaning of a
proposition $A$ is given by explaining what constitutes a proof of
$A$, and proofs of a logically compound $A$ are explained in terms of
proofs of its constituents.  A version of this interpretation is
described in \cite[Chapter 1]{VDT}.

While Kolmogorov's and Heyting's work were independent of each other,
they both acknowledged similarities between the calculus of problems
and the proof interpretation.  However, they regarded these respective
interpretations as distinct.  Later, in 1958, Heyting insisted that
the two interpretations are practically the same and also extended
them to predicate calculus.  Since then, the two interpretations have
been treated as the same and are widely known as the
\emph{Brouwer/Heyting/Kolmogorov} or \emph{BHK interpretation}.
However, as pointed out in \cite{Diez2000}, there are subtle
differences between the two.

\subsubsection{Other interpretations of intuitionism}

Some other interpretations of intuitionistic propositional and
predicate calculus are as follows:
\begin{itemize}
\item Algebraic semantics, widely known as Heyting algebra semantics,
  were probably first used by Stanis{\l}aw Ja\'skowski in 1936.
\item Topological semantics were implicit in Marshall Harvey Stone's
  work published in 1937 and were introduced explicitly by Alfred
  Tarski in 1938.
\item Beth models were introduced by Evert Willem Beth in 1956.
\item Kripke models were introduced by Saul Aaron Kripke in 1965.
\end{itemize}
These interpretations provide a great many models of intuitionism with
widely varying properties.  Experts will recognize that our Muchnik
topos may be viewed from various perspectives as a Kripke model, a
topological model, and a Heyting algebra model.

\subsection{Higher-order logic and sheaf semantics}
\label{subsec:logic and sheaves}

\subsubsection{Higher-order logic}

\emph{Higher-order logic} is a kind of logic where, in addition to
quantifiers over objects, one has quantifiers over pairs of objects,
sets of objects/pairs, sets of sets of objects/pairs, functions from
objects to objects, functions from functions to functions, and so on.
This augmentation of so-called first-order logic increases its
expressive power and is a useful framework for certain foundational
studies.

Higher-order logic calls for a \emph{many-sorted} or \emph{typed}
language.  In Subsection \ref{subsec:higher-order intuitionistic
  logic} below, we provide a detailed definition of the language of
higher-order logic.  This language together with appropriate axioms
and rules of inference is sufficiently rich to permit the development
of virtually all of intuitionistic mathematics.

\subsubsection{Sheaf semantics for intuitionistic higher-order logic}

\emph{Sheaf theory} originated in the mid-20th century in a
geometrical context.  Subsequently it spread to many branches of
mathematics including complex analysis, algebraic geometry, algebraic
topology, differential equations, algebra, category theory,
mathematical logic, and mathematical physics.  Sheaf theory may be
viewed as a general tool which facilitates passage from local
properties to global properties.  For more on the history of sheaf
theory, see Gray \cite{Gray}.

The connection between sheaves and intuitionistic higher-order logic
came from several sources.  An important source was Dana Scott's
topological model of intuitionistic analysis
\cite{scott1968,scott1970}.  Another important source was
\emph{category theory}, an abstract approach to mathematics which was
introduced by Samuel Eilenberg and Saunders Mac Lane in the context of
algebraic topology.  For an introduction to category theory, see
\cite{maclane}.  Alexander Grothendieck and his coworkers gave a
general definition of sheaves over \emph{sites} (rather than merely
over topological spaces) and were thus led to a class of categories
known as \emph{Grothendieck topoi}.  Francis Lawvere realized that
these categories provide enough structure to interpret intuitionistic
higher-order logic.  In collaboration with Myles Tierney, Lawvere
developed the notion of \emph{elementary topoi}, a generalization of
Grothendieck topoi.  For more on sites and Grothedieck topoi, see
\cite{maclane/moerdijk} and \cite[Chapters 14, 15]{VDT}.  For more on
topos theory in general, see \cite{JOHN,lambek/scott}.

In this paper we avoid the complications of category theory and topos
theory.  Instead we follow the sheaf-theoretic approach of Dana Scott,
Michael Fourman, and Martin Hyland
\nocite{SLNM753}\cite{FOURHYL,FOURSCOTT,SCOTT}.  Subsection
\ref{subsec:sheaves} below provides a definition of the category
$\Sh(T)$ of sheaves over a fixed topological space $T$.  Subsection
\ref{subsec:interpretation} explains how to interpret intuitionistic
higher-order logic in $\Sh(T)$.

\subsection{Recursive mathematics and degrees of unsolvability}
\label{subsec:recursion theory}

\subsubsection{Constructive recursive mathematics}
\label{subsubsec:recursive math}

Constructive recursive mathematics (mentioned above in Subsection
\ref{subsec:history}) is a constructivist school that started in the
1930s. It is based on an informal concept of \emph{algorithm} or
\emph{effective procedure}, with the following features.
\begin{itemize}
\item An algorithm is a set of instructions of finite size. The
  instructions themselves are finite strings of symbols from a finite
  alphabet.
\item There is a computing agent (human or machine), which can react
  to the instructions and carry out the computations.
\item The computing agent has unlimited facilities for making,
  storing, and retrieving steps in a computation.
\item The computation is always carried out deterministically in a
  discrete stepwise fashion, without use of continuous methods or
  analog devices. In other words, the computing agent does not need to
  make intelligent decisions or enter into an infinite process at any
  step.
\end{itemize}
On this basis, a $k$-place partial function $f:\subseteq\NN^k\to\NN$
is said to be \emph{effectively calculable} if there is an effective
procedure with the following properties.
\begin{enumerate}
\item Given a $k$-tuple $(m_1,\ldots,m_k)$ in the domain of $f$, the
  procedure eventually halts and returns a correct value of
  $f(m_1,\ldots,m_k)$.
\item Given a $k$-tuple $(m_1,\ldots,m_k)$ not in the domain of $f$,
  the procedure does not halt and does not return a value.
\end{enumerate}

Several formalizations of this informal idea of effectively calculable
functions were developed.  Kurt Friedrich G\"odel used the
\emph{primitive recursive functions} in his famous incompleteness
proof in 1931, and then later introduced \emph{general recursive
  functions} in 1934 following a suggestion of Jacques Herbrand.
Along completely different lines, Alonzo Church introduced the
\emph{$\lambda$-calculus}, a theory formulated in the language of
$\lambda$-abstraction and application, and Haskell Brooks Curry
developed his \emph{combinatory logic}.  The equivalence of
$\lambda$-calculus with combinatory logic was proved by John Barkley
Rosser, Sr.  The equivalence of the Herbrand/G\"odel recursive
functions with the $\lambda$-definable functions was proved by Church
and by Stephen Cole Kleene in 1936.

Alan Turing in 1936--1937 defined an interesting class of algorithms,
now called \emph{Turing machines}, and argued convincingly that the
class of effectively calculable functions coincides with the class of
functions computable by Turing machines.  Independently of Turing,
Emil Leon Post developed a mathematical model for computation in 1936.
The \emph{Church/Turing thesis}, also known as Church's thesis, states
that for each of the above formalisms, the class of functions
generated by the formalism coincides with the informally defined class
of effectively calculable functions.  This was proposed in 1936 and is
now almost universally accepted, although no formal proof is possible,
because of the non-rigorous nature of the informal definition of
effective calculability.

The study of constructive recursive mathematics was continued by
Markov and his students.  Again, the functions computable by Markov
algorithms were shown to be the same as the Herbrand/G\"odel recursive
functions and the Turing computable functions.  Markov's approach to
recursive mathematics was constructive, but he explicitly accepted the
following consequence of PEM:
\begin{quote}
  ``If it is impossible that an algorithmic computation does not
  terminate, then it does terminate.''
\end{quote}
This principle, known as \emph{Markov's principle}, was rejected by
the intuitionists.  We comment further on Markov's principle in
Subsection \ref{subsec:choice-poset} below.

As noted in \cite{Troelstra_historyof}, the discovery of precise
definitions of effective calculability and the Church/Turing thesis in
the 1930's had no effect on the philosophical basis of intuitionism.
Each of these definitions describes algorithms in terms of a specific
language, which is contrary to Brouwer's view of mathematics as the
languageless activity of the ideal mathematician.  Turing's analysis
is not tied to a specific formalism, but his arguments are based on
manipulation of symbols and appeals to physical limitations on
computing.  Such arguments are incompatible with Brouwer's idea of
mathematics as a free creation.

Our discussion above is based on \cite[Chapter 1]{VDT} and on
\cite{enderton,rogers,Troelstra_historyof}.

\subsubsection{Unsolvable problems and Turing degrees}

A convincing example of a function which is not effectively calculable
was given by Turing in 1936 via the \emph{halting problem}.  Turing
proved that there is no Turing machine program which decides whether
or not a given Turing machine program will eventually halt. This was
the first example of an unsolvable decision problem.  Soon afterward,
many other mathematical decision problems were shown to be unsolvable,
for instance Hilbert's 10th problem (the problem of deciding whether a
given Diophantine equation has a solution in integers) and the word
problem for groups.

Eventually it became desirable to compare the amounts of unsolvability
inherent in various unsolvable problems.  Informally and vaguely, a
problem $A$ is said to be \emph{solvable relative to} a problem $B$ if
there exists a Turing algorithm which provides a solution of $A$ given
a solution of $B$.  If in addition $B$ is not solvable relative to
$A$, then $B$ is strictly more unsolvable than $A$, i.e., problem $B$
has a strictly greater \emph{degree of unsolvability} than problem
$A$.

The concept of oracle machines, described by Turing in 1939, gave a
means of comparing unsolvable problems.  In 1944 Emil Post introduced
the rigorous notion of \emph{Turing reducibility} and \emph{Turing
  degrees} as a formalization of degrees of unsolvability associated
with decision problems.  It was shown that the Turing degrees form an
upper semi-lattice, i.e., a partially ordered set in which any finite
set has a least upper bound.  See \cite{enderton,rogers} and
Subsection \ref{subsec:Muchnik topos} below.

\subsubsection{Mass problems}

In order to formalize Kolmogorov's calculus of problems, Yu.\
T. Medvedev \cite{medvedev} introduced mass problems.  A \emph{mass
  problem} is a subset of the Baire space $\NN^\NN=\{f\mid
f:\NN\to\NN\}$.  A mass problem is identified with its set of
solutions.  Informally, to ``solve'' a mass problem $P$ means to
``find'' or ``construct'' an element of the set
$P\subseteq\NN^\NN$. Formally, if $P$ and $Q$ are mass problems, $P$
is said to be \emph{strongly reducible} or \emph{Medvedev reducible}
to $Q$, written $P\les Q$, if there exists an effectively calculable
partial functional from the Baire space to itself which maps each
element of $Q$ to some element of $P$.  It can be shown that $\les$ is
a reflexive and transitive relation on the powerset of $\NN^\NN$. The
\emph{strong degree} or \emph{Medvedev degree} of a mass problem $P$,
denoted $\degs(P)$, is the equivalence class consisting of all mass
problems $Q$ which are \emph{strongly equivalent} to $P$, i.e., $P\les
Q$ and $Q\les P$.  Following Kolmogorov's ideas \cite{KOL-German}
concerning the calculus of problems, Medvedev proved rigorously that
the collection of all strong degrees, denoted $\Ds$, is a model of
intuitionistic propositional calculus.

Later Albert Abramovich Muchnik \cite{MUCH}\footnote{An English
  translation of Muchnik's paper is included as an appendix to this
  paper.} introduced a variant notion of reducibility for mass
problems, known as \emph{weak reducibility} or \emph{Muchnik
  reducibility}.  A mass problem $P$ is said to be weakly reducible to
a mass problem $Q$, written $P\lew Q$, if for each $g\in Q$ there
exists an effectively calculable partial functional which maps $g$ to
some $f\in P$.  Again, $\lew$ is a reflexive and transitive relation
on the powerset of $\NN^\NN$.  The \emph{weak degree} or \emph{Muchnik
  degree} of a mass problem $P$, denoted $\degw(P)$, is the
equivalence class consisting of all mass problems $Q$ which are
\emph{weakly equivalent} to $P$, i.e., $P\lew Q$ and $Q\lew P$. Still
following Kolmogorov \cite{KOL-German}, Muchnik proved that the
collection of all weak degrees, denoted $\Dw$, is a model of
intuitionistic propositional calculus.

Thus each of $\Dw$ and $\Ds$ provides a rigorous implementation of
Kolmogorov's non-rigorous calculus of problems.  Muchnik \cite{MUCH}
says that the difference between weak and strong reducibility of mass
problems is analogous to the difference between proving the existence
of a solution of a differential equation versus effectively finding
such a solution.

Subsection \ref{subsec:Muchnik topos} below provides further details
on mass problems and Muchnik degrees.  For a more extensive
discussion, see \cite{masseff}.  For more on Muchnik degrees and their
relationship to intuitionistic propositional calculus, see
\cite{Kuyper2013,SIMPSON5,SOTER1}.

The purpose of this paper is to extend Muchnik's interpretation of
intuitionistic propositional calculus to intuitionistic higher-order
logic.  The extension is defined in terms of a sheaf model based on
the Muchnik degrees.  We call our sheaf model \emph{the Muchnik
  topos}.  We feel that our study of the Muchnik topos helps to
strengthen the connection between two important subjects, intuitionism
and degrees of unsolvability.

There is another line of research, known as \emph{realizability},
which was initiated by Kleene in 1945 \cite{Kleene1945} and further
developed by other researchers, especially Martin Hyland and Jaap van
Oosten.  Both the realizability interpretation and our Muchnik topos
interpretation provide close connections between intuitionism and
recursion theory.  However, these two interpretations are quite
different.  One may draw the following analogy:
\[
\frac{\mbox{Medvedev reducibility}}{\mbox{realizability
    topos}}=\frac{\mbox{Muchnik reducibility}}{\mbox{Muchnik topos}}.
\]
For a historical account and survey of realizability, see
\cite{vanOosten2002}.  For recent work on the realizability topos, see
\cite{vanOosten2013,vanOosten2008}.

\subsection{Outline of this paper}

The plan of this paper is as follows.

Sections \ref{sec:sheaves and logic} through \ref{sec:Numbers} consist
of background material concerning sheaf models and intuitionism.  In
Section \ref{sec:sheaves and logic} we describe \emph{sheaves}
(a.k.a., sheaves of sets) over topological spaces, and we explain how
the sheaves over any fixed topological space form a model of
intuitionistic higher-order logic.  In Section \ref{sec:Sh(K)} we
discuss sheaf models over topological spaces of a particular kind,
namely, \emph{poset spaces}.  We also discuss a \emph{choice
  principle} which fails in some sheaf models but which holds in sheaf
models over poset spaces and over the Baire space.  In Section
\ref{sec:Numbers} we explain how the various number systems and the
Baire space are standardly represented as sheaves within sheaf models
of intuitionistic mathematics.

The heart of this paper is Section \ref{sec:Muchnik topos}.  In
Subsection \ref{subsec:Muchnik topos} we review the definitions of
Turing degrees and Muchnik degrees, and we note that Muchnik degrees
can be identified with upwardly closed sets of Turing degrees.  We
then define the \emph{Muchnik topos} to be the sheaf model over the
poset of Turing degrees.  In Subsection \ref{subsec:Muchnik reals} we
introduce a new representation of the intuitionistic real number
system, which we call the \emph{Muchnik reals}.  The idea is that a
Muchnik real ``comes into existence'' only when we have enough Turing
oracle power to compute it.  In Subsection \ref{subsec:ACBP} we prove
a choice principle and a bounding principle for the Muchnik reals.
Thus it emerges that intuitionistic analysis based on the Muchnik
reals bears some formal similarity to recursive analysis.

In an Appendix we provide an English translation of Muchnik's paper
\cite{MUCH}.  This is the paper where Muchnik defined the Muchnik
degrees and used them to interpret intuitionistic propositional
calculus along the lines which had suggested by Kolmogorov.  This
paper \cite{MUCH} is important for us, because our Muchnik topos
interpretation may be viewed as a natural extension of Muchnik's
interpretation, from intuitionistic propositional calculus
\cite[Section 1]{MUCH} to intuitionistic mathematics as a whole.

\section{Sheaves and intuitionistic higher-order logic}
\label{sec:sheaves and logic}

In this section we provide background material on sheaf theory and
intuitionistic higher-order logic.  Our main references are
\cite{basu} and \cite[Chapter 14]{VDT}.

\subsection{Sheaves over a topological space}
\label{subsec:sheaves}

\begin{dfn}\label{dfn:-sheaf}
  Let $T$ be a topological space.  Let $\Omega=\{U\subseteq T\mid U
  \hbox{ is open}\}$.  A \emph{sheaf} over $T$ is an ordered triple
  $M=(M,E_M,\ser_M)$ (we omit the subscripts on $E$ and $\ser$ when
  there is no chance of confusion), where $M$ is a set and $E$ and
  $\ser$ are functions, $E:M\to\Omega$ and $\ser\,:M\times\Omega\to
  M$, with the following properties.
  \begin{enumerate}
  \item\label{sheaf1}$a\ser E(a)=a$ for all $a\in M$.
  \item\label{sheaf2}$E(a\ser U)=E(a)\cap U$ for all $a\in M$ and all
    $U\in\Omega$.
  \item\label{sheaf3}$(a\ser U)\ser V=a\ser(U\cap V)$ for all $a\in M$
    and all $U,V\in\Omega$.
  \item\label{sheaf4}$M$ is partially ordered by letting $a\le b$ if
    and only if $a=b\ser E(a)$.
  \item\label{sheaf5}Say that $a,b\in M$ are \emph{compatible} if
    $a\ser E(b)=b\ser E(a)$. Say that $C\subseteq M$ is
    \emph{compatible} if the elements of $C$ are pairwise
    compatible. Then, any compatible set $C\subseteq M$ has a
    \emph{supremum} or least upper bound with respect to $\le$,
    denoted $\sup C$. That is, for all $d\in M$ we have $\sup C\le d$
    if and only if $a\le d$ for all $a\in C$.
  \end{enumerate}
  Elements of a sheaf $M$ are called \emph{sections} of $M$. A
  \emph{global section} is a section $a$ such that $E(a)=T$. The
  operations $E$ and $\ser$ are called \emph{extent} and
  \emph{restriction} respectively.  Thus, for any $a\in M$ and
  $U\in\Omega$, $E(a)\in\Omega$ is the \emph{extent} of $a$ and $a\ser
  U\in M$ is the \emph{restriction} of $a$ to $U$.
\end{dfn}

\begin{exa}\label{exa:-partialcontinuous}
  A good example of a sheaf over $T$ is
  \[
  \Co(T,X)=\{a:\dom(a)\to X\mid\dom(a)\in\Omega,\, a\hbox{ is
    continuous}\}
  \]
  where $X$ is any topological space, with $E$ and $\ser$ given by
  $E(a)=\dom(a)=$ the domain\footnote{For any function $a$ we write
    $\dom(a)=$ the domain of $a$, and $\rng(a)=$ the range of $a$.} of
  $a$, and $a\ser U=a\res U=$ the restriction of $a$ to
  $U\cap\dom(a)\in\Omega$.
\end{exa}

\begin{exa}
  $\Omega$ itself is a sheaf over $T$, with $E(U)=U$ and $U\ser
  V=U\cap V$ for all $U,V\in\Omega$.  Note that
  $\Omega\cong\Co(T,\{0\})$ where $\{0\}$ is the one-point space.
\end{exa}

\begin{exa}\label{exa:-pomega}
  Let $T$ and $\Omega$ be as in Definition \ref{dfn:-sheaf}.  We
  define
  \[
  \Omegai=\{(V,U)\mid V,U\in\Omega,\,V\subseteq U\}
  \]
  with $E$ and $\ser$ given by $E((V,U))=U$ and $(V,U)\ser W=(V\cap
  W,U\cap W)$ for all $(V,U)\in\Omegai$ and all $W\in\Omega$.  It can
  be shown that $\Omegai$ is a sheaf over $T$.  In fact,
  $\Omegai\cong\Co(T,S)$ where $S$ is the \emph{Sierpi\'nski space},
  i.e., the topological space $\{0,1\}$ with open sets
  $\emptyset,\{1\},\{0,1\}$.  For details see \cite[pages
  16--17]{basu}.
\end{exa}

\begin{dfn}
  Let $(M,E_M,\ser_M)$ and $(N,E_N,\ser_N)$ be sheaves over $T$. We
  say that $(N,E_N,\ser_N)$ is a \emph{subsheaf} of $(M,E_M,\ser_M)$
  if $N\subseteq M$ and $E_N$ and $\ser_N$ are inherited from $M$,
  i.e., $E_N(a)=E_M(a)$ and $a\ser_N U=a\ser_M U$ for all $a\in N$ and
  all $U\in\Omega$.
\end{dfn}

\begin{exa}\label{exa:-locallyconstantpartiallycontinuous}
  Let $X$ be a topological space.  A function $a$ from a subset of $T$
  into $X$ is said to be \emph{locally constant} if for every
  $t\in\dom(a)$ there exists an open set $V\in\Omega$ such that $t\in
  V$ and $a$ is constant on $V\cap\dom(a)$.  Clearly locally constant
  functions are continuous.  Let
  \[
  \Colc(T,X)=\{a\in\Co(T,X)\mid a\hbox{ is locally constant}\}.
  \]
  Then $\Colc(T,X)$ is a subsheaf of $\Co(T,X)$.  However,
  \[
  \Coc(T,X)=\{a\in\Co(T,X)\mid a\hbox{ is constant}\}
  \]
  is in general not a sheaf, hence not a subsheaf of $\Colc(T,X)$.
\end{exa}

\begin{lem}\label{lem:-Eresjoin}
  Let $M$ be a sheaf over $T$.
  \begin{enumerate}
  \item For all $a,b\in M$, if $a\le b$ then $E(a)\subseteq E(b)$.
  \item If $C$ is a compatible subset of $M$, then $c=\sup C$ if and
    only if $E(c)=\bigcup_{a\in C}E(a)$ and $a\le c$ for all $a\in C$.
  \item If $\{U_i\mid i\in I\}$ is a family of open subsets of $T$,
    and if $c\in M$, then $\{c\ser U_i\mid i\in I\}$ is a compatible
    subset of $M$ and $\sup_{i\in I}(c\ser U_i)=c\ser\bigcup_{i\in
      I}U_i$.
  \item Every bounded subset of $M$ is compatible and has a least
    upper bound.
  \end{enumerate}
\end{lem}

\begin{proof}
  The proof is straightforward.  See \cite[pages 13--16]{basu}.
\end{proof}

\begin{dfn}
  Let $M$ and $N$ be sheaves over $T$.
  \begin{enumerate}
  \item The \emph{product sheaf} is
    \[
    M\times N=\{(a,b)\mid a\in M,\,b\in N,\,E(a)=E(b)\}
    \]
    with $E$ and $\ser$ given by $E((a,b))=E(a)$ and $(a,b)\ser
    U=(a\ser U,b\ser U)$.
  \item For all $U\in\Omega$ the \emph{restriction sheaf} is
    \[
    M\ser U=\{a\ser U\mid a\in M\}=\{a\in M\mid E(a)\subseteq U\},
    \]
    with $E$ and $\ser$ inherited from $M$.
  \item A \emph{sheaf morphism} $M\stackrel{\varphi}{\to}N$ is a
    mapping $\varphi:M\to N$ satisfying $E(\varphi(a))=E(a)$ and
    $\varphi(a\ser U)=\varphi(a)\ser U$ for all $a\in M$ and all
    $U\in\Omega$.
  \item The \emph{function sheaf} is
    \[
    N^M=\{(\varphi,U)\mid U\in\Omega,\,M\ser
    U\stackrel{\varphi}{\to}N\ser U\}
    \]
    with $E$ and $\ser$ defined by $E((\varphi,U))=U$ and
    $(\varphi,U)\ser V=(\varphi\ser V,U\cap V)$ where $(\varphi\ser
    V)(a)=\varphi(a)\ser V$ for all $a\in M\ser(U\cap V)$ and all
    $V\in\Omega$.
  \end{enumerate}
\end{dfn}

\begin{rem}
  It can be shown that the product sheaf, the restriction sheaf, and
  the function sheaf are indeed sheaves over $T$.  For details see
  \cite[pages 20--28]{basu}.
\end{rem}

\begin{dfn}\label{dfn:-powersheaf1}
  For any sheaf $M$ we define the \emph{power sheaf} $\pow(M)$ to be
  the function sheaf $\Omegai^M$ where $\Omegai$ is as in Example
  \ref{exa:-pomega}.
\end{dfn}

\begin{rem}
  An alternative definition of the power sheaf appears in
  \cite{FOURSCOTT,VDT}.  It can be shown that this alternative
  definition is equivalent to our Definition \ref{dfn:-powersheaf1}.
  For details see \cite[pages 28--35]{basu}.
\end{rem}

\begin{thm}
  For any sheaf $M$ there is a natural one-to-one correspondence
  between the subsheaves of $M$ and the global sections of $\pow(M)$.
\end{thm}

\begin{proof}
  See \cite[pages 35--39]{basu}.
\end{proof}

\begin{rem}
  Given a topological space $T$, let $\Sh(T)$ be the category whose
  objects are the sheaves over $T$ and whose morphisms are the sheaf
  morphisms over $T$.  The category $\Sh(T)$ is one of the most basic
  examples of a topos.
\end{rem}

\subsection{The language of higher-order logic}
\label{subsec:higher-order intuitionistic logic}

We now describe a many-sorted language $\lang$ for intuitionistic
higher-order logic.

\begin{dfn}\label{dfn:lang}
  The language $\lang$ is defined as follows.
  \begin{enumerate}
  \item The \emph{sorts} of $\lang$ are generated as follows.
    \begin{enumerate}
    \item There is a collection of \emph{ground
        sorts}.\footnote{A.k.a., \emph{basic sorts} or \emph{primitive
          sorts}.}
    \item If $\sigma$ and $\tau$ are sorts, then so is
      $\sigma\times\tau$, the \emph{product sort} of $\sigma$ and
      $\tau$.
    \item If $\sigma$ and $\tau$ are sorts, then so is
      $\sigma\to\tau$, the \emph{function sort} from $\sigma$ to
      $\tau$.
    \item If $\sigma$ is a sort, then so is $P\sigma$, the \emph{power
        sort} of $\sigma$.
    \end{enumerate}
  \item The symbols of $\lang$ are:
    \begin{enumerate}
    \item for each sort $\sigma$, an infinite supply of
      \emph{variables} $x^\sigma,y^\sigma,\ldots$;
    \item for each sort $\sigma$, an \emph{existence predicate}
      $E^\sigma$ of type $(\sigma)$, an \emph{equality predicate}
      $=^\sigma$ of type $(\sigma,\sigma)$ , and a \emph{membership
        predicate} $\in^\sigma$ of type $(\sigma,P\sigma)$;
    \item for all sorts $\sigma$ and $\tau$, a \emph{pairing operator}
      $\pi^{\sigma,\tau}$ of type $(\sigma,\tau,\sigma\times\tau)$ and
      \emph{projection operators} $\pi_1^{\sigma,\tau}$ and
      $\pi_2^{\sigma,\tau}$ of types $(\sigma\times\tau,\sigma)$ and
      $(\sigma\times\tau,\tau)$ respectively, and an \emph{application
        operator} $\Ap^{\sigma,\tau}$ of type
      $(\sigma\to\tau,\sigma,\tau)$;
    \item \emph{propositional connectives}
      $\lnot,\land,\lor,\limp,\liff$;
    \item \emph{quantifiers} $\forall,\exists$.
    \end{enumerate}
    When there is no danger of confusion, we may omit superscripts
    indicating sorts and types.
  \item The \emph{terms} of $\lang$ are generated as follows.
    \begin{enumerate}
    \item Each variable of sort $\sigma$ is a term of sort $\sigma$.
    \item If $s$ and $t$ are terms of sort $\sigma$ and $\tau$
      respectively, and if $\pi$ is of type
      $(\sigma,\tau,\sigma\times\tau)$, then $\pi st$ is a term of
      sort $\sigma\times\tau$.
    \item If $r$ is a term of sort $\sigma\times\tau$, and if
      $\pi_1,\pi_2$ are of type $(\sigma\times\tau,\sigma)$ and
      $(\sigma\times\tau,\tau)$ respectively, then $\pi_1r$ and
      $\pi_2r$ are terms of sort $\sigma$ and $\tau$ respectively.
    \item If $s$ and $t$ are terms of sort $\sigma$ and
      $\sigma\to\tau$ respectively, and if $\Ap$ is of type
      $(\sigma\to\tau,\sigma,\tau)$, then $\Ap ts$ is a term of sort
      $\tau$.  We usually write $ts$ instead of $\Ap ts$.
    \end{enumerate}
  \item The \emph{atomic formulas} of $\lang$ are:
    \begin{enumerate}
    \item $r=s$, where $r$ and $s$ are terms of sort $\sigma$ and $=$
      is of type $(\sigma,\sigma)$;
    \item $s\in t$, where $s$ and $t$ are terms of sort $\sigma$ and
      $P\sigma$ respectively, and $\in$ is of type $(\sigma,P\sigma)$;
    \item $Es$, where $s$ is a term of sort $\sigma$ and $E$ is of
      type $(\sigma)$.
    \end{enumerate}
  \item The \emph{formulas} of $\lang$ are generated as follows.
    \begin{enumerate}
    \item Each atomic formula is a formula.
    \item If $A,B$ are formulas then so are $\lnot A$, $A\land B$,
      $A\lor B$, $A\limp B$, $A\liff B$.
    \item If $A$ is a formula and $x$ is a variable, then $\forall
      x\,A$ and $\exists x\,A$ are formulas.
    \end{enumerate}
  \end{enumerate}
\end{dfn}

\begin{dfn}
  The \emph{complexity} of a formula $A$ is the number of occurrences
  of propositional connectives $\lnot,\land,\lor,\limp,\liff$ and
  quantifiers $\forall,\exists$ in $A$.  An occurrence of a variable
  $x^\sigma$ in $A$ is said to be \emph{bound in $A$} if it is within
  the scope of a quantifier $\forall x^\sigma$ or $\exists x^\sigma$
  in $A$. A variable $x^\sigma$ is said to be \emph{free in $A$} if at
  least one occurrence of $x^\sigma$ in $A$ is not bound in $A$.  A
  formula $A$ is called a \emph{sentence} if no variables are free in
  $A$.
\end{dfn}

\begin{rem}
  For a more extensive discussion, see \cite{VDAL1}.  We could include
  additional predicates and operators in $\lang$, but they are not
  needed for our purpose.
\end{rem}

\subsection{Sheaf models of intuitionistic mathematics}
\label{subsec:interpretation}

In this subsection we explain how sheaves over topological spaces
provide models of intuitionistic higher-order logic and intuitionistic
mathematics.

\begin{dfn}\label{dfn:-lang(mu)}
  Let $T$ be a topological space.  Let $\mu$ be a mapping which
  assigns to each ground sort $\sigma$ of $\lang$ a sheaf $M_\sigma$
  over $T$.  We inductively extend $\mu$ to the compound sorts of
  $\lang$ by letting $M_{\sigma\times\tau}=M_\sigma\times M_\tau$
  (product sheaf), $M_{\sigma\to\tau}=M_\tau^{M_\sigma}$ (function
  sheaf), and $M_{P\sigma}=\pow(M_\sigma)$ (power sheaf).  For each
  sort $\sigma$ of $\lang$ and each section $a\in M_\sigma$, we extend
  $\lang$ by adding a constant symbol $a=a^\sigma$ of sort $\sigma$,
  which is now also a term of sort $\sigma$.  The extended language is
  denoted $\lang(\mu)$.  A term of $\lang(\mu)$ is said to be
  \emph{closed} if it contains no variables.
\end{dfn}

\begin{dfn}\label{dfn:llbrrb}
  To each $\lang(\mu)$-sentence $A$ we assign a \emph{truth value}
  $\llb A\rrb\in\Omega$.
  \begin{enumerate}
  \item To each closed $\lang(\mu)$-term $s$ of sort $\sigma$, we
    assign a \emph{value} $\llb s\rrb\in M_\sigma$.
    \begin{enumerate}
    \item If $a\in M_\sigma$ let $\llb a\rrb=a$.
    \item If $s$ and $t$ are closed terms of sorts $\sigma$ and $\tau$
      respectively, let $\llb\pi st\rrb=(\llb s\rrb\ser E(\llb
      t\rrb),\llb t\rrb\ser E(\llb s\rrb))$.
    \item If $r$ is a closed term of sort $\sigma\times\tau$, then
      $\llb r\rrb=(a,b)$ for some $(a,b)\in M_\sigma\times M_\tau$ and
      we let $\llb\pi_1r\rrb=a$ and $\llb\pi_2r\rrb=b$.
    \item Suppose $t$ is a closed term of sort $\sigma\to\tau$ with
      $\llb t\rrb=(\varphi,U)\in M_\tau^{M_\sigma}$. If $s$ is a
      closed term of sort $\sigma$, let $\llb\Ap ts\rrb=\llb
      ts\rrb=\varphi(\llb s\rrb\ser U)$.
    \end{enumerate}
  \item For atomic $\lang(\mu)$-sentences $A$, we define $\llb
    A\rrb\in\Omega$ as follows.
    \begin{enumerate}
    \item If $r$ and $s$ are closed terms of sort $\sigma$, let
      \[
      \llb r=s\rrb=\bigcup\{U\in\Omega\mid U\subseteq E(\llb
      r\rrb)\cap E(\llb s\rrb),\,\llb r\rrb\ser U=\llb s\rrb\ser U\}.
      \]
    \item If $s$ is a closed term of sort $\sigma$, let $\llb E^\sigma
      s\rrb=\llb s=s\rrb=E(\llb s\rrb)$.
    \item If $s$ is a closed term of sort $\sigma$ and $t$ is a closed
      term of sort $P\sigma$ with $\llb
      t\rrb=(\varphi,U)\in\pow(M_\sigma)=\Omegai^{M_\sigma}$, let
      $\llb s\in t\rrb=V$ where $\varphi(\llb s\rrb\ser U)=(V,E(\llb
      s\rrb)\cap U)$.
    \end{enumerate}
  \item For non-atomic $\lang(\mu)$-sentences $A$, we define $\llb
    A\rrb\in\Omega$ by induction on the complexity of $A$, using the
    notation $S^\circ=$ interior of $S$.
    \begin{enumerate}
    \item Propositional connectives:
      \[
      \begin{array}{l}
        \llb\lnot A\rrb=(T\setminus\llb A\rrb)^\circ,\\[8pt]
        \llb A\land B\rrb=\llb A\rrb\cap\llb B\rrb,\quad
        \llb A\lor B\rrb=\llb A\rrb\cup\llb B\rrb,\\[8pt]
        \llb A\limp B\rrb=(\llb A\rrb\limp\llb B\rrb)
        \hbox{ where }(U\limp V)=((T\setminus U)\cup
        V)^\circ,\\[8pt]
        \llb A\liff B\rrb=\llb A\limp B\rrb\cap\llb B\limp A\rrb.
      \end{array}
      \]
    \item Quantifiers:
      \[
      \begin{array}{l}
        \llb\exists
        x^\sigma\,A(x^\sigma)\rrb=\displaystyle\bigcup_{a\in
          M_\sigma}\llb Ea\land A(a)\rrb,\\[16pt]
        \llb\forall
        x^\sigma\,A(x^\sigma)\rrb=\left(\displaystyle\bigcap_{a\in
            M_\sigma}\llb Ea\limp A(a)\rrb\right)^\circ.
      \end{array}
      \]
    \end{enumerate}
  \end{enumerate}
\end{dfn}

\begin{dfn}\label{dfn:Sh(T,mu)}
  For $\lang(\mu)$-sentences $A$ we write $\Sh(T,\mu)\models A$ to
  mean that $\llb A\rrb=T$.  An $\lang$-formula $A$ is said to be
  \emph{valid for sheaf models} if for all topological spaces $T$ and
  all $\mu:\sigma\mapsto M_\sigma$ as above, $\Sh(T,\mu)\models$ the
  universal closure of $A$.
\end{dfn}

The following theorem says that the axioms and rules of intuitionistic
higher-order logic are valid for sheaf models.  Let $\IHOL$ be the
formal system of intuitionistic higher-order logic as formulated in
\cite{SCOTT} and \cite[Chapter 14]{VDT}.
\begin{thm}\label{thm:-IHOL}
  The axioms and rules of $\IHOL$ are valid for sheaf models.
\end{thm}

\begin{proof}
    See \cite[Theorem 7.3]{FOURSCOTT} and \cite[Theorem 5.15]{VDT}.
\end{proof}

For instance, substitution of equals is intuitionistically valid,
hence provable in $\IHOL$, so we have:
\begin{thm}\label{thm:-subeq}
  Let $x$ be a variable of sort $\sigma$, let $r$ and $s$ be closed
  $\lang(\mu)$-terms of sort $\sigma$, and let $A(x)$ be an
  $\lang(\mu)$-formula with no free variables other than $x$.  Then
  $\Sh(T,\mu)\models r=s\limp(A(r)\liff A(s))$, hence $\llb
  r=s\rrb\cap\llb A(r)\rrb\subseteq\llb A(s)\rrb$.
\end{thm}

\begin{proof}
  For a much more detailed proof, see \cite[pages 42--48]{basu}.
\end{proof}

\begin{rem}
  By \cite{FOURSCOTT,VDT} we know that intuitionistic mathematics is
  formalizable in $\IHOL$.  Thus Theorem \ref{thm:-IHOL} may be viewed
  as saying that, for any topological space $T$, $\Sh(T)$ is a model
  of intuitionistic mathematics.  Such models are known as \emph{sheaf
    models}.
\end{rem}

\section{Poset spaces and choice principles}
\label{sec:Sh(K)}

In this section we discuss sheaf models over a special class of
topological spaces, the so-called \emph{poset spaces}.  We show that
some special cases of the axiom of choice are valid for sheaf models
over poset spaces and over the Baire space.

\subsection{Poset spaces}
\label{subsec:sheaves on posets}

\begin{dfn}
  A \emph{poset}\footnote{I.e., a partially ordered set.} is a
  non-empty set $K$ together with a binary relation ${\le}$ on $K$
  which is reflexive, antisymmetric, and transitive.  A set
  $U\subseteq K$ is said to be \emph{upwardly closed} if for all
  $\alpha\in U$ and $\beta\in K$, $\alpha\le\beta$ implies $\beta\in
  U$.  The upwardly closed subsets of $K$ are the open sets of a
  topology on $K$, the \emph{Alexandrov topology}.  A \emph{poset
    space} is a poset endowed with the Alexandrov topology.  The
  category of sheaves over a poset space $K$ is denoted $\Sh(K)$.
\end{dfn}

\begin{lem}\label{lem:-Alexandrov}
  Let $K$ be a poset space.
  \begin{enumerate}
  \item For any $\alpha\in K$ there is a smallest open set containing
    $\alpha$, namely,
    \[
    U_\alpha=\{\beta\in K\mid\alpha\le\beta\}.
    \]
  \item $K$ is \emph{locally connected}, i.e., for any $\alpha\in K$
    and any neighborhood $U$ of $\alpha$, there is a connected
    neighborhood of $\alpha$ included in $U$.
  \item For all families of subsets of $K$ we have
    $\left(\bigcap_{i\in I}S_i\right)^\circ=\bigcap_{i\in
      I}S_i^\circ$.
  \item If $U$ is an open subset of $K$, and if $X$ is a $T_1$
    space\footnote{A \emph{$T_1$ space} is a topological space in
      which every point is a closed set.  Examples of $T_1$ spaces are
      $\RR$, $\NN$, $\NN^\NN$, etc.}, then every continuous function
    $f:U\to X$ is locally constant.
  \end{enumerate}
\end{lem}

\begin{proof}
  The proof is straightforward.  See \cite[pages 50--51]{basu}.
\end{proof}

\begin{dfn}
  A poset $K$ is said to be \emph{directed} if for all
  $\alpha,\beta\in K$ there exists $\gamma\in K$ such that
  $\alpha\le\gamma$ and $\beta\le\gamma$.
\end{dfn}

\begin{lem}\label{lem:-DirPosetLem}
  Let $U$ be an upward closed subset in a directed poset $K$.  Let $X$
  be any set.  Then, any locally constant function $f:U\to X$ is
  constant.
\end{lem}

\begin{proof}
  The proof is straightforward.  See \cite[pages 51--52]{basu}.
\end{proof}

\subsection{Choice principles over poset spaces}
\label{subsec:choice-poset}

In this subsection we show that sheaf models over poset spaces satisfy
certain special cases of the axiom of choice.

\begin{dfn}\label{dfn:-AC}
  Let $\sigma$ and $\tau$ be $L$-sorts.  The \emph{axiom of choice}
  for $\sigma\to\tau$, denoted $\AC(\sigma,\tau)$, is the universal
  closure of
  \begin{center}
    $(\forall x\,\exists y\,A(x,y))\limp\exists w\,\forall x\,A(x,wx)$
  \end{center}
  where $x,y,w$ are variables of sort $\sigma,\tau,\sigma\to\tau$
  respectively, and $A(x,y)$ is any $\lang$-formula in which $w$ does
  not occur.
\end{dfn}

\begin{rem}
  A model of intuitionistic higher-order logic cannot satisfy
  $\AC(\sigma,\tau)$ for all sorts $\sigma,\tau$ unless it is also a
  model of classical higher-order logic.  This is because, as shown in
  \cite{DIAC}, the full axiom of choice implies PEM.  However, as we
  shall see, models such as $\Sh(T,\mu)$ may satisfy
  $\AC(\sigma,\tau)$ for some particular choices of $\sigma$ and
  $\tau$.
\end{rem}

\begin{dfn}\label{dfn:-Xsh}
  Let $T$ be a topological space, and let $X$ be a set.  As in Example
  \ref{exa:-locallyconstantpartiallycontinuous}, let $\Colc(T,X)$ be
  the sheaf of locally constant functions from open subsets of $T$
  into $X$.  We define $\Xsh=\Colc(T,X)$.  Note that for each $x\in X$
  there is a global section $\xhat$ of $\Xsh$ which maps $T$ into
  $\{x\}$.  The sheaf $\Xsh$ is called a \emph{simple sheaf}.  See
  \cite{FOURSCOTT} and \cite[page 782]{VDT}.
\end{dfn}

\begin{thm}\label{thm:-ADCK}
  Let $K$ be a poset space.  If $M_\sigma$ is a simple sheaf over $K$,
  then $\Sh(K,\mu)$ satisfies $\AC(\sigma,\tau)$.
\end{thm}

\begin{proof}
  We may safely assume that $A(x,y)$ has no free variables other than
  $x$ and $y$.  Letting $U=\llb\forall x\,\exists y\,A(x,y)\rrb$, it will
  suffice to show that $U\subseteq\llb\exists w\,\forall
  x\,A(x,wx)\rrb$.

  Let $X$ be a set such that $M_\sigma=\Xsh$.  For each $x\in X$ we
  have $\xhat\in\Xsh$ and $E(\xhat)=K$, hence
  \[
  \begin{array}{rl}
    U&=\displaystyle\left(\bigcap_{a\in
        \Xsh}\left(E(a)\limp\llb\exists
        y\,A(a,y)\rrb\right)\right)^\circ\\
    &\subseteq\displaystyle\bigcap_{x\in
      X}\left(E(\xhat)\limp\llb\exists
      y\,A(\xhat,y)\rrb\right)\\
    &=\displaystyle\bigcap_{x\in X}\llb\exists
    y\,A(\xhat,y)\rrb\qquad\hbox{(because }E(\xhat)=K)\\
    &=\displaystyle\bigcap_{x\in X}\bigcup_{b\in
    M_\tau}(E(b)\cap\llb A(\xhat,b)\rrb).
  \end{array}
  \]

  Fix $\alpha\in U$.  Using the axiom of choice externally, we choose
  for each $x\in X$ a $b_x\in M_\tau$ such that $\alpha\in
  E(b_x)\cap\llb A(\xhat,b_x)\rrb$.  Since $E(b_x)\cap\llb
  A(\xhat,b_x)\rrb$ is open, it follows by Lemma \ref{lem:-Alexandrov}
  that $U_\alpha\subseteq E(b_x)\cap\llb A(\xhat,b_x)\rrb$.  We shall
  now define a sheaf morphism from $\Xsh\ser U_\alpha$ into $M_\tau$.
  Let $a\in \Xsh\ser U_\alpha$ be given.  For all $x,y\in X$ such that
  $x\ne y$ we have $a^{-1}(x)\cap a^{-1}(y)=\emptyset$, hence the set
  $\{b_x\ser a^{-1}(x)\mid x\in X\}\subseteq M_\tau$ is compatible,
  hence the least upper bound
  \[
  \varphi(a)=\displaystyle\sup_{x\in X}\big(b_x\ser a^{-1}(x)\big)\in
  M_\tau
  \]
  exists.  We have $E(a)\subseteq U_\alpha\subseteq E(b_x)$, hence
  $E(\varphi(a))=\bigcup_{x\in X}E(b_x\ser
  a^{-1}(x))=\bigcup_{x\in X}a^{-1}(x)=E(a)$ by Lemma
  \ref{lem:-Eresjoin}.  Moreover, for any open set $V\subseteq K$ we
  have $\varphi(a\ser V)=\sup_{x\in X}\big(b_x\ser(a\ser
  V)^{-1}(x)\big)=\sup_{x\in X}\big(b_x\ser\big(a^{-1}(x)\cap
  V\big)\big)=\left(\sup_{x\in X}\big(b_x\ser
    a^{-1}(x)\big)\right)\ser V=\varphi(a)\ser V$.  Thus $\varphi$
  preserves extent and restriction, so we have a sheaf morphism
  \[
  \Xsh\ser U_\alpha\stackrel{\varphi}{\to}M_\tau\ser U_\alpha,
  \]
  i.e., $(\varphi,U_\alpha)\in M_\tau^{\Xsh}$.

  We claim that $U_\alpha\cap E(a)\subseteq\llb
  A(a,(\varphi,U_\alpha)a)\rrb$ for all $a\in\Xsh$.  To see this, fix
  $\beta\in U_\alpha\cap E(a)$.  For some $x\in X$ we have
  $a(\beta)=x$, hence $U_\beta\subseteq a^{-1}(x)\subseteq E(a)$ and
  $\xhat\ser U_\beta=a\ser U_\beta$, hence
  $U_\beta\subseteq\llb\xhat=a\rrb$.  Moreover $b_x\ser
  U_\beta=\varphi(a\ser U_\alpha)\ser U_\beta$ and $U_\beta\subseteq
  E(b_x)$, hence $U_\beta\subseteq\llb b_x=\varphi(a\ser
  U_\alpha)\rrb$, and clearly $U_\beta\subseteq U_\alpha\cap
  E(a)=\llb\varphi(a\ser U_\alpha)=(\varphi,U_\alpha)a\rrb$.
  Therefore, from $U_\beta\subseteq\llb A(\xhat,b_x)\rrb$ it follows
  by Theorem \ref{thm:-subeq} that $U_\beta\subseteq\llb
  A(a,(\varphi,U_\alpha)a)\rrb$, and this proves the claim.

  Our claim easily implies that
  \[
  \begin{array}{rcl}
    U_\alpha&\subseteq&\displaystyle\bigcap_{a\in \Xsh}\big((K\setminus
    E(a))\cup\llb A(a,(\varphi,U_\alpha)a)\rrb\big)^\circ\\[20pt]
    &=&\llb\forall x\,A(x,(\varphi,U_\alpha)x)\rrb.
  \end{array}
  \]
  But then, since $E((\varphi,U_\alpha))=U_\alpha$, we have
  \[
  \begin{array}{rcl}
    U_\alpha&\subseteq&E((\varphi,U_\alpha))\cap\llb\forall
    x\,A(x,(\varphi,U_\alpha)x)\rrb\\[8pt]
    &\subseteq&\displaystyle\bigcup_{(\psi,V)\in
      M_\tau^{\Xsh}}(E((\psi,V)\cap\llb\forall
    x\,A(x,(\psi,V)x)\rrb)\\[20pt]
    &=&\llb\exists w\,\forall x\,A(x,wx)\rrb.
  \end{array}
  \]
  Since $\alpha\in U$ was arbitrary, we conclude that
  $U\subseteq\llb\exists w\,\forall x\,A(x,wx)\rrb$.  This completes
  the proof of Theorem \ref{thm:-ADCK}.
\end{proof}

\begin{rem}\label{rem:-ADCR}
  Theorem \ref{thm:-ADCK} fails for sheaf models over arbitrary
  topological spaces.  In particular, see \cite[pages 77--79]{basu}
  and \cite[page 788]{VDT} for a proof that $\AC(\sigma,\sigma)$ fails
  in $\Sh(\RR,\mu)$ for $M_\sigma=\NNsh$.  See also Remark
  \ref{rem:notAC0} below.
\end{rem}

\begin{rem}\label{rem:subXsh}
  One might think that Theorem \ref{thm:-ADCK} should hold whenever
  $M_\sigma$ is a subsheaf of a simple sheaf over $K$.  However, the
  following example shows otherwise.  Let $K=\{-\infty\}\cup\{-i\mid
  i\in\NN\}$ with the natural linear ordering, $-\infty<-j<-i$ for all
  $i,j\in\NN$ with $i<j$.  Let $M_\sigma=\Co(K,\{0\})\ser\{-i\mid
  i\in\NN\}$.  For all $m,n\in\NN$ let $b_{m,n}\in\Co(K,\NN)$ be the
  constant function with domain $\{-i\mid i<m\}$ and value $n$.  Let
  $M_\tau$ be the subsheaf of $\Co(K,\NN)$ consisting of all $b_{m,n}$
  such that $m<n$.  Note that for each sheaf morphism $(\varphi,U)\in
  M_{\sigma\to\tau}$ we have $U\subseteq\{-i\mid i\in\NN\}$.  Let
  $x,y,w$ be variables of sort $\sigma,\tau,\sigma\to\tau$
  respectively.  Easy calculations show that $\llb\forall x\,\exists
  y\,(y=y)\rrb=K$ and $\llb\exists w\,\forall x\,(wx=wx)\rrb=\{-i\mid
  i\in\NN\}$.  Thus $\AC(\sigma,\tau)$ fails in $\Sh(K)$ for the
  formula $A(x,y)\equiv(y=y)$.
\end{rem}

In the vein of Theorem \ref{thm:-ADCK} and Remark \ref{rem:-ADCR}, we
now call attention to another principle which is valid for sheaf
models over poset spaces but not over arbitrary topological spaces.

\begin{dfn}\label{dfn:GMP}
  Let $\GMP(\sigma)$ be the universal closure of
  \[
  \left(\forall x\,(A(x)\lor\lnot A(x))\land\lnot\lnot\exists
    x\,A(x)\right)\limp\exists x\,A(x)
  \]
  where $x$ is a variable of sort $\sigma$ and $A(x)$ is any
  $\lang$-formula.  As will become clear in Subsection
  \ref{subsec:natural numbers}, $\GMP(\sigma)$ for $M_\sigma=\NNsh$
  amounts to \emph{Markov's principle} as discussed in \cite[page
  203]{VDT} and in Subsection \ref{subsubsec:recursive math} above.
  Thus $\GMP(\sigma)$ may be viewed as a generalized Markov principle.
\end{dfn}

\begin{thm}
  \label{thm:-GMPK}
  Let $K$ be a poset space.  If $M_\sigma$ is a simple sheaf over $K$,
  then $\Sh(K,\mu)$ satisfies $\GMP(\sigma)$.
\end{thm}

\begin{proof}
  See \cite[pages 71--74]{basu}.
\end{proof}

\begin{rem}
  Neither Markov's principle nor its generalization in Theorem
  \ref{thm:-GMPK} holds for sheaf models over arbitrary topological
  spaces $T$.  In fact, Markov's principle fails over $T=\{0,1\}^\NN=$
  the Cantor space.  See \cite[pages 69--71]{basu}.
\end{rem}

\subsection{Choice principles over the Baire space}
\label{subsec:choice-Baire}

Despite Remark \ref{rem:-ADCR}, Theorem \ref{thm:-ADCK} is valid for
for sheaf models over some topological spaces other than poset spaces.
We now show that the Baire space $\NN^\NN$ is one such topological
space.

\begin{lem}\label{lem:-CountableDisjointSubcover}
  Given a collection $\calU$ of open sets in $\NN^\NN$, we can find a
  collection $\calV$ of open sets in $\NN^\NN$ such that
  \begin{enumerate}
  \item $\bigcup\calV=\bigcup\calU$,
  \item for all $V\in\calV$ there exists $U\in\calU$ such that
    $V\subseteq U$, and
  \item for all $V,V'\in\calV$, if $V\ne V'$ then $V\cap
    V'=\emptyset$.
  \end{enumerate}
\end{lem}

\begin{proof}
  For each finite sequence $p$ of natural numbers, let
  $V_p=\{f\in\NN^\NN\mid p$ is an initial segment of $f\}$.  Given
  $\calU$ as in the lemma, let $\calV=\{V_p\mid p$ minimal such that
  $\exists U\,(U\in\calU$ and $V_p\subseteq U)\}$.  Clearly $\calV$
  has the desired properties.
\end{proof}

\begin{thm}\label{thm:-ADCB}
  If $M_\sigma$ is a simple sheaf over $\NN^\NN$, then
  $\Sh(\NN^\NN,\mu)\models\AC(\sigma,\tau)$.
\end{thm}

\begin{proof}
  Let $X$ be a set such that $M_\sigma=\Xsh$.  As in the proof of
  Theorem \ref{thm:-ADCK}, let $U=\llb\forall x\,\exists
  y\,A(x,y)\rrb$ and note that
  \[
  U\subseteq\displaystyle\bigcap_{x\in X}\bigcup_{b\in
    M_\tau}(E(b)\cap\llb A(\xhat,b)\rrb).
  \]
  For each $x\in X$ apply Lemma \ref{lem:-CountableDisjointSubcover}
  to get a pairwise disjoint collection $\calV_x$ of open sets such
  that $U\subseteq\bigcup\calV_x$ and for all $V\in\calV_x$ there
  exists $b\in M_\tau$ such that $V\subseteq E(b)\cap\llb
  A(\xhat,b)\rrb$.  For each $V\in\calV_x$ choose such a $b$ and let
  $b_{x,V}=b\ser V$.  Clearly $\{b_{x,V}\mid V\in\calV_x\}\subseteq
  M_\tau$ is compatible, so for all $a\in\Xsh\ser U$ define
  \[
  \varphi(a)=\sup_{x\in X}\sup_{V\in\calV_x}\big(b_{x,V}\ser
  a^{-1}(x)\big)\in M_\tau.
  \]
  The verification that $\Xsh\ser U\stackrel{\varphi}{\to}M_\tau\ser
  U$ is a sheaf morphism, the proof that $U\cap E(a)\subseteq\llb
  A(a,(\varphi,U)a)\rrb$ for all $a\in\Xsh$, and the final
  verification that $U\subseteq\llb\exists w\,\forall x\,A(x,wx)\rrb$,
  are similar to the corresponding parts of the proof of Theorem
  \ref{thm:-ADCK}.  For further details, see \cite[pages
  80--83]{basu}.
\end{proof}

\begin{rem}
  A different proof of Theorem \ref{thm:-ADCB} for the special case
  $M_\sigma=\NNsh$ is given in \cite[page 289]{FOURHYL} and \cite[page
  787]{VDT}.
\end{rem}

\begin{rem}\label{rem:-ADCB}
  Our proof of Theorem \ref{thm:-ADCB} uses only the property of the
  Baire space which is stated in Lemma
  \ref{lem:-CountableDisjointSubcover}.  Therefore, Theorem
  \ref{thm:-ADCB} holds for sheaf models over all topological spaces
  with this property.
\end{rem}

\section{Sheaf representations of the number systems}
\label{sec:Numbers}

Let $T$ be a topological space.  In this section we discuss the
representation of the number systems $\NN,\QQ,\RR$ and the Baire space
$\NN^\NN$ within the sheaf model $\Sh(T)$.

\subsection{The natural numbers}
\label{subsec:natural numbers}

Recall from Subsection \ref{subsec:choice-poset} that
$\NNsh=\Colc(T,\NN)$ where $\NN$ is the set of natural numbers.  In
this subsection we argue that $\NNsh$ is appropriately viewed as
representing the natural number system within $\Sh(T)$.

\begin{dfn}\label{dfn:-Peano}
  A \emph{system} is an ordered triple $(X,c,f)$ where $X$ is a set,
  $c\in X$, and $f:X\to X$.  A \emph{Peano system} is a system which
  satisfies $\forall x\,(fx\ne c)$ and $\forall x\,\forall
  y\,(f(x)=f(y)\limp x=y)$ and
  \[
  \forall Y\,((Y\subseteq X\land c\in Y\land\forall x\,(x\in Y\limp
  f(x)\in Y))\limp Y=X).
  \]
\end{dfn}

\begin{thm}\label{thm:-Peano}
  The following familiar facts are intuitionistically valid.
  \begin{enumerate}
  \item Given a Peano system $(X,c,f)$ and a system $(X',c',f')$,
    there is a unique $h:X\to X'$ satisfying $h(c)=c'$ and $\forall
    x\,(h(f(x))=f'(h(x)))$.
  \item Any two Peano systems are isomorphic.
  \item In any Peano system $(X,f,c)$ there are uniquely determined
    functions satisfying the primitive recursion equations,
    identifying $c$ with $0$ and $f$ with the successor function
    $S$. In particular, there are uniquely determined operations $+$
    and $\cdot$ on $X$ satisfying
    \[
    \begin{array}{ll}
      x+c=x,&x+f(y)=f(x+y),\\[8pt]
      x\cdot c=c,&x\cdot f(y)=(x\cdot y)+x
    \end{array}
    \]
    for all $x,y\in X$.
  \end{enumerate}
\end{thm}

\begin{proof}
  See \cite[Chapter 3]{VDT}.
\end{proof}

\begin{rem}
  We interpret Definition \ref{dfn:-Peano} and Theorem
  \ref{thm:-Peano} in $\Sh(T)$ by letting $X$ be a sheaf over $T$, $c$
  a section of $X$, $f:X\to X$ a sheaf morphism, and $Y$ a section of
  the power sheaf $P(X)$.  By Theorem \ref{thm:-IHOL} we know that
  Theorem \ref{thm:-Peano} is valid in $\Sh(T)$.  Therefore, the
  following theorem implies that $\NNsh$ is the ``correct''
  representation of the natural number system as a sheaf over $T$.
\end{rem}

\begin{thm}\label{thm:-natural numbers}
  Let $T$ be a topological space.  Let $\widehat0\in\NNsh$ be the
  global section given by $\widehat0(t)=0$ for all $t\in T$.  Let
  $\widehat{S}:\NNsh\to\NNsh$ be the sheaf morphism given by
  $(\widehat{S}(a))(t)=a(t)+1$ for all $a\in\NNsh$ and all $t\in
  E(a)$.  Then $\Sh(T)$ satisfies that $(\NNsh,\widehat0,\widehat{S})$
  is a Peano system.
\end{thm}

\begin{proof}
  A detailed proof is in \cite[pages 58--61]{basu}.
\end{proof}

\begin{rem}
  Similarly, the sheaves in $\Sh(T)$ corresponding to $\ZZ$, the ring
  of integers, and $\QQ$, the field of rational numbers, are $\ZZsh$
  and $\QQsh$ respectively.  See also \cite[Chapter III]{FOURSCOTT}
  and \cite[Chapter 15]{VDT}.
\end{rem}

\subsection{The Baire space}
\label{subsec:Baire space}

In this subsection we discuss the representation of the Baire space
$\NN^\NN$ within $\Sh(T)$.  We begin by noting that, since the simple
sheaf $\NNsh=\Colc(T,\NN)$ represents $\NN$, the function sheaf
$\widehat{\NN}^{\sh\widehat{\NN}^\sh}$ represents $\NN^\NN$.

\begin{thm}\label{thm:-Baire space}
  For any topological space $T$, $\Co(T,\NN^\NN)$ and
  $\widehat{\NN}^{\sh\widehat{\NN}^\sh}$ are isomorphic as sheaves
  over $T$.  Hence $\Co(T,\NN^\NN)$ represents $\NN^\NN$ within
  $\Sh(T)$.
\end{thm}

\begin{proof}
  For a detailed proof, see \cite[pages 62--64]{basu}.
\end{proof}

\begin{thm}\label{thm:-B sp for loc conn T}
  If $T$ is locally connected, then $\Co(T,\NN^\NN)=\Colc(T,\NN^\NN)$,
  so the simple sheaf $\widehat{\NN^\NN}^\sh=\Colc(T,\NN^\NN)$
  represents $\NN^\NN$ within $\Sh(T)$.
\end{thm}

\begin{proof}
  Let $U\subseteq T$ be open.  Given a continuous function
  $a:U\to\NN^\NN$, for each $i\in\NN$ define a continuous function
  $a_i:U\to\NN$ by $a_i(t)=(a(t))(i)$.  If $U$ is connected, then each
  $a_i$ is constant on $U$, hence $a$ is constant on $U$.  Since $T$
  is locally connected, it follows that
  $\Co(T,\NN^\NN)=\Colc(T,\NN^\NN)$.  Therefore, by Definition
  \ref{dfn:-Xsh} and Theorem \ref{thm:-Baire space},
  $\Colc(T,\NN^\NN)=\widehat{\NN^\NN}^\sh$ represents $\NN^\NN$ in
  $\Sh(T)$.
\end{proof}

\begin{cor}
  For any poset space $K$, the simple sheaf
  $\widehat{\NN^\NN}^\sh=\Colc(K,\NN^\NN)$ represents $\NN^\NN$ within
  $\Sh(K)$.
\end{cor}

\begin{proof}
  By Lemma \ref{lem:-Alexandrov} $K$ is locally connected, so Theorem
  \ref{thm:-B sp for loc conn T} applies to $K$.
\end{proof}

\begin{rem}
  Theorems \ref{thm:-Baire space} and \ref{thm:-B sp for loc conn T}
  for $\NN^\NN$ hold more generally, for product spaces $X^Y$ where
  $X$ has the discrete topology.  In other words, over any topological
  space $T$ the sheaves $\widehat{X}^{\sh\widehat{Y}^\sh}$ and
  $\Co(T,X^Y)$ are isomorphic, and if $T$ is locally connected then
  $\Co(T,X^Y)=\Colc(T,X^Y)=\widehat{X^Y}^\sh$.
\end{rem}

\subsection{The real numbers}
\label{subsec:real numbers}

In classical mathematics, the \emph{Cauchy reals} (real numbers
constructed as equivalence classes of Cauchy sequences of rational
numbers) and the \emph{Dedekind reals} (real numbers constructed as
Dedekind cuts of rational numbers) are equivalent.
Intuitionistically, they are not necessarily equivalent.  In this
subsection we discuss various sheaf models where they are and are not
equivalent.

\begin{dfn}
  Classically, we use $\RR$ to denote the real number system.
  Intuitionistically, we use $\RR_C$ and $\RR_D$ to denote the Cauchy
  reals and the Dedekind reals respectively.  In particular, given a
  topological space $T$, we use $\RR_C$ and $\RR_D$ to denote the
  sheaves in $\Sh(T)$ corresponding to the Cauchy reals and the
  Dedekind reals respectively.  Recall from Subsection
  \ref{subsec:sheaves} that $\Co(T,X)$ (respectively $\Colc(T,X)$,
  $\Coc(T,X)$) are the sheaves of continuous (respectively locally
  constant, constant) functions from open subsets of $T$ into $X$.  If
  $M$ is any one these sheaves over $T$, there is a natural
  isomorphism of $\QQsh=\Co(T,\QQ)=\Colc(T,\QQ)$ onto a subsheaf of
  $M$, corresponding to the natural embedding of $\QQ$ into $\RR$.  If
  $M_1$ and $M_2$ are any two of these sheaves, we say that $M_1$ and
  $M_2$ are \emph{$\QQ$-isomorphic}, denoted $M_1\congQ M_2$, if there
  is an isomorphism of $M_1$ onto $M_2$ which commutes with the
  natural embeddings of $\QQsh$ into $M_1$ and $M_2$.
\end{dfn}

\begin{thm}\label{thm:-DedRCauR}
  Let $T$ be a topological space.  Within $\Sh(T)$ we have
  $\RR_D\congQ\Co(T,\RR)$.  Moreover, if $T$ is locally connected then
  $\RR_C\congQ\Colc(T,\RR)$.
\end{thm}

\begin{proof}
  See \cite[pages 288--289]{FOURHYL}, \cite[pages
  384--385]{FOURSCOTT}, and \cite[pages 784--789]{VDT}.
\end{proof}

\begin{cor}\label{cor:-DedRneqCauR}
  In $\Sh(\RR)$ we have $\RR_C\not\congQ\RR_D$.
\end{cor}

\begin{proof}
  $\RR$ is locally connected, so $\RR_C\congQ\Colc(\RR,\RR)$ and
  $\RR_D\congQ\Co(\RR,\RR)$.  On the other hand, there are continuous
  real-valued functions on $\RR$ which are not locally constant, e.g.,
  the identity function on $\RR$.  Thus
  $\Colc(\RR,\RR)\subsetneqq\Co(\RR,\RR)$, and from this it follows
  easily that $\Colc(\RR,\RR)\not\congQ\Co(\RR,\RR)$.
\end{proof}

\begin{cor}\label{cor:RCRDK}
  Let $K$ be a poset space.  In $\Sh(K)$ we have
  $\RR_C\congQ\RR_D\congQ\Colc(K,\RR)=\Co(K,\RR)$.  Moreover,
  $\Sh(K,\mu)$ satisfies $\AC(\sigma,\tau)$ for $M_\sigma=\RR_C$.
\end{cor}

\begin{proof}
  By Lemma \ref{lem:-Alexandrov} $K$ is locally connected and all
  continuous functions from open subsets of $K$ into $\RR$ are locally
  constant.  Thus Theorem \ref{thm:-DedRCauR} implies that
  $\RR_C\congQ\RR_D\congQ\Co(K,\RR)=\Colc(K,\RR)=\RRsh$.  Theorem
  \ref{thm:-ADCK} tells us that $\Sh(K,\mu)$ satisfies
  $\AC(\sigma,\tau)$ for $M_\sigma=\RRsh$, but since
  $\RR_C\congQ\RRsh$ we get the same conclusion for $M_\sigma=\RR_C$.
\end{proof}

\begin{cor}\label{cor:-DirPoset}
  Let $K$ be a directed poset space.  In $\Sh(K)$ we have
  $\RR_C\congQ\RR_D\congQ\Coc(K,\RR)=\Colc(K,\RR)=\Co(K,\RR)$.
\end{cor}

\begin{proof}
  This follows from the previous corollary plus Lemma
  \ref{lem:-DirPosetLem}.
\end{proof}

\begin{dfn}\label{dfn:-AC_0}
  The \emph{axiom of countable choice} is the special case
  $M_\sigma=\NNsh$ of $\AC(\sigma,\tau)$ as formulated in Definition
  \ref{dfn:-AC}.  More formally, for any topological space $T$ we say
  that $\Sh(T)$ satisfies $\AC_0$ if
  $\Sh(T,\mu)\models\AC(\sigma,\tau)$ for $M_\sigma=\NNsh$ and
  arbitrary $M_\tau$.
\end{dfn}

\begin{thm}\label{thm:-AC_0-Cauchy=Dedekind}
  Let $T$ be a topological space.  If $\Sh(T)$ satisfies $\AC_0$, then
  $\Sh(T)$ satisfies $\RR_C\congQ\RR_D$.
\end{thm}

\begin{proof}
  It is known intuitionistically that the axiom of countable choice
  implies that the Cauchy reals and the Dedekind reals are isomorphic
  over $\QQ$.  Therefore, by Theorem \ref{thm:-IHOL}, this implication
  holds in $\Sh(T)$.  See also \cite[page 289]{FOURHYL} and
  \cite[pages 274 and 788--789]{VDT}.
\end{proof}

\begin{cor}\label{cor:-CoTR=ColcTR}
  Let $T$ be a locally connected topological space.  If $\Sh(T)$
  satisfies $\AC_0$ then $\Co(T,\RR)=\Colc(T,\RR)$.
\end{cor}

\begin{proof}
  This is immediate from Theorems \ref{thm:-DedRCauR} and
  \ref{thm:-AC_0-Cauchy=Dedekind}.
\end{proof}

\begin{rem}\label{rem:notAC0}
  We noted in Remark \ref{rem:-ADCR} that $\AC_0$ fails in $\Sh(\RR)$.
  Now Corollaries \ref{cor:-DedRneqCauR} and \ref{cor:-CoTR=ColcTR}
  provide another proof of this fact.
\end{rem}

\begin{thm}\label{thm:-AC_0KB}
  $\AC_0$ and $\RR_C\congQ\RR_D$ hold in $\Sh(\NN^\NN)$ and in
  $\Sh(K)$ for any poset space $K$.
\end{thm}

\begin{proof}
  This is immediate from Theorems \ref{thm:-ADCK}, \ref{thm:-ADCB},
  \ref{thm:-AC_0-Cauchy=Dedekind}, and \ref{thm:-AC_0KB}.
\end{proof}

\begin{rem}
  There are continuous functions from $\NN^\NN$ into $\RR$ which are
  not locally constant.  Thus $\NN^\NN$ is an example of a topological
  space $T$ such that in $\Sh(T)$ we have
  $\RR_C\congQ\RR_D\congQ\Co(T,\RR)\supsetneqq\Colc(T,\RR)$, hence
  $\RR_C\not\congQ\Colc(T,\RR)$.
\end{rem}

\section{The Muchnik topos and the Muchnik reals}
\label{sec:Muchnik topos}

In this section we discuss a particular sheaf model which we call
\emph{the Muchnik topos}.  We show that the Muchnik topos provides a
model of intuitionistic mathematics which is a natural extension of
the well known Kolmogorov/Muchnik interpretation of intuitionistic
propositional calculus via mass problems under weak reducibility,
i.e., Muchnik degrees.  Within the Muchnik topos we define a sheaf
representation of the real number system which we call \emph{the
  Muchnik reals}.  We prove a choice principle and a bounding
principle for the Muchnik reals.

\subsection{The Muchnik topos}
\label{subsec:Muchnik topos}

\begin{dfn}
  For $f,g\in\NN^\NN$ we say that $f$ is \emph{Turing reducible} to
  $g$, denoted $f\leT g$, if $f$ is computable using $g$ as a Turing
  oracle.  It can be shown that $\leT$ is transitive and reflexive on
  $\NN^\NN$.  We say that $f$ is \emph{Turing equivalent} to $g$,
  denoted $f\eqT g$, if $f\leT g$ and $g\leT f$.  Clearly $\eqT$ is an
  equivalence relation on $\NN^\NN$.  The \emph{Turing degree} of $f$,
  denoted $\degT(f)$, is the equivalence class of $f$ under $\eqT$.
  The set of all Turing degrees is denoted $\DT$.  We partially order
  $\DT$ by letting $\degT(f)\le\degT(g)$ if and only if $f\leT g$.
\end{dfn}

\begin{lem}
  Some well known facts about the poset $\DT$ are as follows.
  \begin{enumerate}
  \item There is a bottom Turing degree $\ooo=\degT(f)$ for computable
    $f\in\NN^\NN$.
  \item Any two Turing degrees have a \emph{supremum}, i.e., a least
    upper bound, given by $\sup(\degT(f),\degT(g))=\degT((f,g))$ where
    $(f,g)\in\NN^\NN$ is given by $(f,g)(2i)=f(i)$ and
    $(f,g)(2i+1)=g(i)$ for all $i\in\NN$.
  \item However, two incomparable Turing degrees may or may not have
    an \emph{infimum}, i.e., a greatest lower bound, in $\DT$.
  \item Thus $\DT$ is an upper semi-lattice, hence a directed poset,
    but not a lattice.
  \end{enumerate}
\end{lem}

\begin{dfn}
  A \emph{mass problem} is a set $P\subseteq\NN^\NN$.  For
  $P,Q\subseteq\NN^\NN$ we say that $P$ is \emph{weakly reducible} to
  $Q$, denoted $P\lew Q$, if for all $g\in Q$ there exists $f\in P$
  such that $f\leT g$.  Clearly $\lew$ is reflexive and transitive on
  the powerset of $\NN^\NN$. We say that $P$ is \emph{weakly
    equivalent} to $Q$, denoted $P\eqw Q$, if $P\lew Q$ and $Q\lew P$.
  Clearly $\eqw$ is an equivalence relation on the power set of
  $\NN^\NN$.  The \emph{weak degree} or \emph{Muchnik degree} of a
  mass problem $P$, denoted $\degw(P)$, is the equivalence class of
  $P$ under $\eqw$.  The set of all Muchnik degrees is denoted $\Dw$.
  We partially order $\Dw$ by letting $\degw(P)\le\degw(Q)$ if and
  only if $P\lew Q$.
\end{dfn}

\begin{rem}
  There is a natural embedding of the Turing degrees, $\DT$, into the
  Muchnik degrees, $\Dw$, given by $\degT(f)\mapsto\degw(\{f\})$.
  This embedding is one-to-one and \emph{order-preserving}, i.e.,
  $f\leT g$ if and only if $\{f\}\lew\{g\}$.  Moreover, this embedding
  preserves the bottom Turing degree and the supremum of any two
  Turing degrees. However, it does not preserve the infimum of two
  incomparable Turing degrees, even when the infimum exists.
\end{rem}

\begin{dfn}
  A \emph{lattice} is a poset such that for any two elements $\aaa$
  and $\bbb$ there exists a \emph{supremum} or least upper bound,
  $\sup(\aaa,\bbb)$, and an \emph{infimum} or greatest lower bound,
  $\inf(\aaa,\bbb)$.  A lattice is said to be \emph{complete} if for
  every set of elements $\{\aaa_i\}_{i\in I}$ there exists a
  \emph{supremum} or least upper bound, $\sup_{i\in I}\aaa_i$, and an
  \emph{infimum} or greatest lower bound, $\inf_{i\in I}\aaa_i$.  Note
  that every complete lattice has a top element and a bottom element.
  A complete lattice is said to be \emph{completely distributive} if
  it satisfies $\inf(\sup_{i\in I}\aaa_i,\bbb)=\sup_{i\in
    I}\inf(\aaa_i,\bbb)$ and $\sup(\inf_{i\in
    I}\aaa_i,\bbb)=\inf_{i\in I}\sup(\aaa_i,\bbb)$ for all
  $\{\aaa_i\}_{i\in I}$ and all $\bbb$.
\end{dfn}

\begin{rem}
  Our reference for lattice theory is Birkhoff, second edition
  \cite{BIRK}.  Our reason for preferring the second edition to the
  third edition is explained in \cite[Remark 1.5]{SIMPSON5}.
\end{rem}

\begin{dfn}
  A set $U\subseteq\DT$ is said to be \emph{upwardly closed} if
  $\degT(g)\in U$ whenever $\degT(f)\in U$ for some $f\leT g$.  Let
  $\calU(\DT)$ be the set of upwardly closed subsets of $\DT$. We
  partially order $\calU(\DT)$ by the subset relation: $U\le V$ if and
  only if $U\subseteq V$.  Clearly $\calU(\DT)$ is a complete and
  completely distributive lattice.  To prove this, one uses only the
  fact that $\DT$ is a poset.
\end{dfn}

\begin{thm}\label{thm:-Turwead}
  The posets $\Dw$ and $\calU(\DT)$ are dually isomorphic.  That is,
  there is an order-reversing one-to-one correspondence between $\Dw$
  and $\calU(\DT)$.
\end{thm}

\begin{proof}
  Define $\Psi:\calU(\DT)\to\Dw$ by letting
  $\Psi(U)=\degw(\{f\mid\degT(f)\in U\})$ for all $U\in\calU(\DT)$.
  It is straightforward to verify that $\Psi$ is one-to-one, onto, and
  \emph{order-reversing}, i.e., $U\subseteq V$ if and only if
  $\Psi(U)\ge\Psi(V)$.  In proving these properties, one uses only the
  fact that $\leT$ is reflexive and transitive on $\NN^\NN$.
\end{proof}

\begin{cor}
  $\Dw$ is a complete and completely distributive lattice.  The
  lattice operations in $\Dw$ are given by
  \[
  \begin{array}{c}
    \sup(\aaa,\bbb)=\Psi(\Psi^{-1}(\aaa)\cap\Psi^{-1}(\bbb))\,,\quad
    \inf(\aaa,\bbb)=\Psi(\Psi^{-1}(\aaa)\cup\Psi^{-1}(\bbb))\,,\\[6pt]
    \displaystyle\sup_{i\in I}\aaa_i=\Psi\left(\bigcap_{i\in
        I}\Psi^{-1}(\aaa_i)\right),\quad
    \displaystyle\inf_{i\in I}\aaa_i=\Psi\left(\bigcup_{i\in
        I}\Psi^{-1}(\aaa_i)\right)
  \end{array}
  \]
  where $\Psi$ is as in the proof of Theorem \ref{thm:-Turwead}.
  Moreover, the top degree in $\Dw$ is
  $\mathbf{\infty}=\Psi(\emptyset)=\degw(\emptyset)$ and the bottom
  degree in $\Dw$ is $\ooo=\Psi(\DT)=\degw(\{f\mid f$ is
  computable$\})$.
\end{cor}

\begin{dfn}
  We define the \emph{Muchnik topos} to be the sheaf model $\Sh(\DT)$.
  Here $\DT$ is a poset space as usual, with the Alexandrov topology,
  where the open sets are the upward closed subsets of $\DT$.
\end{dfn}

\begin{rem}
  Our terminology ``the Muchnik topos'' is motivated by Theorem
  \ref{thm:-Turwead}.  Note that set $\Omega$ of truth values in
  $\Sh(\DT)$ is just $\calU(\DT)$.  Moreover, the propositional
  connectives of Section \ref{sec:sheaves and logic} correspond via
  $\Psi$ to lattice operations in the Muchnik lattice $\Dw$.  Namely,
  for all $\lang(\mu)$-sentences $A$ and $B$, letting $\Psi(\llb
  A\rrb)=\aaa$ and $\Psi(\llb B\rrb)=\bbb$ we have
  \[
  \begin{array}{l}
    \Psi(\llb A\land B\rrb)=\sup(\aaa,\bbb),\quad
    \Psi(\llb A\lor B\rrb)=\inf(\aaa,\bbb),\\[8pt]
    \Psi(\llb A\limp
    B\rrb)=\imp(\aaa,\bbb)=\inf\{\ccc\mid\sup(\aaa,\ccc)\ge\bbb\},\\[8pt]
    \Psi(\llb\lnot A\rrb)=\imp(\aaa,\mathbf{\infty}).
  \end{array}
  \]
  Moreover, $\Sh(\DT,\mu)\models A\limp B$ if and only if
  $\aaa\ge\bbb$, i.e., $\llb A\rrb\subseteq\llb B\rrb$, and
  $\Sh(\DT,\mu)\models B$ if and only if $\bbb=\ooo$, i.e., $\llb
  B\rrb=\DT$.  Thus the Muchnik topos $\Sh(\DT)$ provides a natural
  extension of Muchnik's $\Dw$ interpretation of intuitionistic
  propositional calculus \cite[Section 1]{MUCH} to intuitionistic
  higher-order logic.  See also our translation of \cite{MUCH} in the
  Appendix below.
\end{rem}

\begin{thm}\label{thm:-ADCDT}
  The Muchnik topos $\Sh(\DT)$ satisfies $\AC(\sigma,\tau)$ whenever
  $M_\sigma$ is a simple sheaf.  In particular, $\Sh(\DT)$ satisfies
  $\AC(\sigma,\tau)$ for $M_\sigma=\RR_C$.
\end{thm}

\begin{proof}
  This follows from Theorem \ref{thm:-ADCK} and Corollary
  \ref{cor:RCRDK} since $\DT$ is a poset.
\end{proof}

\begin{thm}\label{thm:-Rd=Rc_DT}
  In the Muchnik topos $\Sh(\DT)$, the Cauchy reals $\RR_C$ and the
  Dedekind reals $\RR_D$ are $\QQ$-isomorphic to each other and to
  $\Co(\DT,\RR)=\Colc(\DT,\RR)=\Coc(\DT,\RR)$.
\end{thm}

\begin{proof}
  This follows from Corollary \ref{cor:-DirPoset} because $\DT$ is a
  directed poset.
\end{proof}

\subsection{The Muchnik reals}
\label{subsec:Muchnik reals}

\begin{dfn}
  Let $\#:\QQ\to\NN$ be a standard G\"odel numbering of the rational
  numbers.  For instance, we could define $\#(q)$ for $q\in\QQ$ by
  \[
  \#(q)=\left\{
    \begin{array}{lll}
      1&\hbox{ if }&q=0,\\[4pt]
      2\cdot3^a\cdot5^b&\hbox{ if
      }&q=a/b\hbox{ where }a,b\in\NN\setminus\{0\}\hbox{ and
      }\gcd(a,b)=1,\\[4pt]
      4\cdot3^a\cdot5^b&\hbox{ if
      }&q=-a/b\hbox{ where }a,b\in\NN\setminus\{0\}\hbox{ and
      }\gcd(a,b)=1.
    \end{array}
  \right.
  \]
  For real numbers $x\in\RR$ we define the \emph{Turing degree of $x$}
  to be $\degT(x)=\degT(f_x)$ where $f_x\in\NN^\NN$ is given by
  $f_x(i)=1$ if $i=\#(q)$ for some $q\in\QQ$ such that $q<x$,
  otherwise $f_x(i)=0$, for all $i\in\NN$.
\end{dfn}

\begin{dfn}
  In $\Sh(\DT)$, the \emph{Muchnik reals} are the sections of the
  sheaf
  \begin{center}
    $\RR_M=\{a\in\Coc(\DT,\RR)\mid\forall\ddd\,(\ddd\in\dom(a)\limp
    \degT(a(\ddd))\le\ddd)\}$.
  \end{center}
  For $a\in\RR_M$ such that $a\ne\emptyset$, let $\abar\in\RR$ be such
  that $\rng(a)=\{\abar\}$, and let $\ahat=$ the maximal $c\in\RR_M$
  such that $a\le c$, i.e., the unique $\ahat\in\RR_M$ such that
  $\rng(\ahat)=\{\abar\}$ and $\dom(\ahat)=$ the Turing upward closure
  of $\{\degT(\abar)\}$.  For $a=\emptyset$ let $\ahat=\emptyset$ and
  let $\abar$ be undefined.  Note that $\emptyset\ne a\le b$ implies
  $\abar=\bbar$ and $\ahat=\bhat$.
\end{dfn}

\begin{rem}
  As we know from Theorem \ref{thm:-Rd=Rc_DT}, the Cauchy reals
  $\RR_C$ and the Dedekind reals $\RR_D$ are represented in $\Sh(\DT)$
  by $\Coc(\DT,\RR)$, the sheaf of constant functions from upward
  closed sets of Turing degrees into $\RR$.  However, not all such
  constant functions are Muchnik reals.  The Muchnik reals are those
  $a\in\Coc(\DT,\RR)$ such that either $a=\emptyset$ or
  $\dom(a)\subseteq$ the upward closure of $\{\degT(\abar)\}$ where
  $\rng(a)=\{\abar\}$.  Thus $\RR_M$ is a proper subsheaf of
  $\Coc(\DT,\RR)$, so $\RR_M\not\congQ\RR_C\congQ\RR_D$.  Informally,
  a Muchnik real is a real number which ``comes into existence'' only
  when we have enough Turing oracle power to compute it.
\end{rem}

\subsection{A bounding principle for the Muchnik reals}
\label{subsec:ACBP}

By Theorem \ref{thm:-ADCDT} the Muchnik topos $\Sh(\DT)$ satisfies a
choice principle for $\RR_C$, the Cauchy reals.  In this subsection we
prove that for $\RR_M$, the Muchnik reals, $\Sh(\DT)$ satisfies not
only a choice principle but also a bounding principle.

\begin{dfn}\label{dfn:-Turing Reducible in Sh(D_T)}
  Let $r,s,t$ be closed terms of sort $\sigma$ where $M_\sigma=\RR_M$.
  Then $a,b,c\in\RR_M$ where $a=\llb r\rrb$, $b=\llb s\rrb$, $c=\llb
  t\rrb$.  We define $\llb r\leT s\rrb=E(a)\cap E(b)$ if
  $a,b\ne\emptyset$ and $\abar\leT\bbar$, otherwise $\llb r\leT
  s\rrb=\emptyset$.  We define $\llb r\leT(s,t)\rrb=E(a)\cap E(b)\cap
  E(c)$ if $a,b,c\ne\emptyset$ and $\abar\leT(\bbar,\cbar)$, otherwise
  $\llb r\leT(s,t)\rrb=\emptyset$.  Our \emph{bounding principle}
  $\BP(\sigma,\sigma)$ for the Muchnik reals is
  \begin{center}
    $(\forall x\,\exists y\,A(x,y))\limp\exists z\,\forall x\,\exists
    y\,(y\leT(x,z)\land A(x,y))$
  \end{center}
  where $x,y,z$ are variables of sort $\sigma$ and $A(x,y)$ is any
  $L$-formula which does not contain $z$.
\end{dfn}

\begin{thm}\label{thm:ACBP}
  The Muchnik topos $\Sh(\DT)$ satisfies a combined \emph{choice and
    bounding principle} $\ACBP(\sigma,\sigma)$ for the Muchnik reals,
  \begin{center}
    $(\forall x\,\exists y\,A(x,y))\limp\exists w\,\exists z\,\forall
    x\,(wx\leT(x,z)\land A(x,wx))$
  \end{center}
  where $x,y,z$ are variables of sort $\sigma$, $w$ is a variable of
  sort $\sigma\to\sigma$, $A(x,y)$ is any $\lang$-formula which
  does not contain $z$ or $w$, and $M_\sigma=\RR_M$.
\end{thm}

\begin{proof}
  We may safely assume that $A(x,y)$ has no free variables other than
  $x$ and $y$.  Letting $U=\llb\forall x\,\exists y\,A(x,y)\rrb$ and
  $V=\llb\exists w\,\exists z\,\forall x\,(wx\leT(x,z)\land
  A(x,wx))\rrb$, it will suffice to show that $U\subseteq V$.  Fix
  $c=\chat\ne\emptyset$ in $\RR_M$ such that $E(c)\subseteq U$.  It
  will suffice to show that $E(c)\subseteq V$.

  For each $a\ne\emptyset$ in $\RR_M$ we have $\degT((\abar,\cbar))\in
  E(\ahat)\cap E(c)=E(\ahat\ser E(c))\subseteq E(c)\subseteq U$, so
  choose $b\in\RR_M$ depending only on $\abar$ such that
  $\degT((\abar,\cbar))\in E(b)\cap\llb A(\ahat\ser E(c),b)\rrb$.  We
  then have $\bbar\leT(\abar,\cbar)$ and $E(b)\supseteq E(\ahat)\cap
  E(c)$, so by Theorem \ref{thm:-subeq} it follows that $E(a)\cap
  E(c)\subseteq\llb b\leT(a,c)\rrb\cap\llb A(a,b)\rrb$.  Moreover,
  since $b$ depends only on $\abar$, we have a sheaf morphism
  \[
  \RR_M\ser E(c)\stackrel{\varphi}{\to}\RR_M\ser E(c)
  \]
  where $\varphi(a\ser E(c))=b\ser E(a)\cap E(c)$ for all $a\in\RR_M$.
  Thus $(\varphi,E(c))\in\RR_M^{\RR_M}$ and
  $\llb(\varphi,E(c))a\rrb=b\ser E(a)\cap E(c)$, so by Theorem
  \ref{thm:-subeq} we have $E(a)\cap
  E(c)\subseteq\llb(\varphi,E(c))a\leT(a,c)\rrb\cap\llb
  A(a,(\varphi,E(c))a)\rrb$.  Since this holds for all $a\in\RR_M$, we
  now see that $E(c)\subseteq V$, and the proof is complete.
\end{proof}

\begin{cor}\label{cor:ACBP}
  $\Sh(\DT)$ satisfies $\AC(\sigma,\sigma)$ and $\BP(\sigma,\sigma)$
  for $M_\sigma=\RR_M$.
\end{cor}

\begin{proof}
  Within the formal system $\IHOL$, $\AC(\sigma,\sigma)$ and
  $\BP(\sigma,\sigma)$ are logical consequences of
  $\ACBP(\sigma,\sigma)$.  Therefore, the corollary follows from
  Theorems \ref{thm:ACBP} and \ref{thm:-IHOL}.  More details may be
  found in \cite[pages 99--106]{basu}.
\end{proof}

\begin{rem}
  In our proof of Theorem \ref{thm:ACBP}, one may avoid using the
  axiom of choice, as follows.  First, given $a\ne\emptyset$ in
  $\RR_M$, let $e_a$ be the smallest index $e\in\NN$ of a partial
  recursive functional $\Phi_e$ such that
  $\Phi_e((\abar,\cbar))=\bbar$ for some $b\in\RR_M$ such that
  $\degT((\abar,\cbar))\in E(b)\cap\llb A(\ahat\ser E(c),b)\rrb$.
  Then, choose $b=\bhat$.
\end{rem}

\begin{thm}\label{thm:-CHPDT}
  $\Sh(\DT)$ satisfies $\AC(\sigma,\tau)$ for $M_\sigma=\RR_M$ and
  $M_\tau$ arbitrary.
\end{thm}

\begin{proof}
  Repeat the proof of Theorem \ref{thm:ACBP} but skip the parts that
  involve $\leT$.
\end{proof}

\begin{rem}
  Theorem \ref{thm:-CHPDT} resembles Theorem \ref{thm:-ADCDT}.
  However, Theorem \ref{thm:-ADCDT} applies only when $M_\sigma$ is a
  simple sheaf, while in Theorem \ref{thm:-CHPDT} we have
  $M_\sigma=\RR_M$ which is not a simple sheaf.  See also Remark
  \ref{rem:subXsh}.
\end{rem}

\newpage 
\phantomsection
\addcontentsline{toc}{section}{References}


\section*{Appendix: translation of Muchnik's paper}
\addcontentsline{toc}{section}{Appendix: translation of Muchnik's paper}
\label{app:translation}

In this appendix we offer a translation of Muchnik's paper
\cite{MUCH}.  We started with a rough translation produced in 1964 by
the United States Department of Commerce \cite{JPRS-translation}.  We
have corrected some typographical and translation errors and updated
some bibliographical references.

\setcounter{footnote}{0}

\begin{center}
  \large Siberian Mathematical Journal\\[2pt]
  \large Vol.\ IV, No. 6, November--December, 1963\\[10pt]
  \large A. A. Muchnik\\[4pt]
  \Large Strong and weak reducibility of algorithmic problems
\end{center}

\subsection*{Introduction}
\addcontentsline{toc}{subsection}{Introduction}

The abstract (arithmetical) analysis of algorithmic problems was
initiated by S. Kleene and E. Post \cite{Post,Kleene-Post}. E. Post
introduced the concept of \emph{degree of unsolvability} of a problem,
while Kleene and Post investigated in \cite{Kleene-Post} the class of
degrees of unsolvability of arithmetical (in the sense of G\"odel)
sets. Papers along the same line were published subsequently.

The traditional algorithmic problems of algebra, number theory,
topology, and mathematical logic were problems of solvability. This
explains the predominant interest shown first in problems of
solvability of arithmetic (i.e., problems of solvability of sets of
natural numbers). Subsequently, however, in logic and its
applications, problems arose connected to separability, enumerability,
and isomorphism of sets
\cite{MedvedevA,Muchnik56,Trahtenbrot53,Uspenskiy}.

The definition of an algorithmic problem in abstract algorithm theory
was formulated by Yu.\ T. Medvedev, in which all the previously known
cases and many others were treated \cite{MedvedevA}.  The problem of
constructing an arithmetical function\footnote{I.e., a function
  defined on the natural numbers $\NN$ and assuming values from $\NN$,
  which includes also $0$.} satisfying certain conditions is called a
\emph{Medvedev problem} (\emph{M-problem}).  To each M-problem $P$
there corresponds a family of functions satisfying the conditions of
the problem. Conversely, any family of functions $A$ defines some
M-problem $P(A)$.  The functions contained in the family corresponding
to an M-problem $P$ are called the \emph{solution functions} of the
M-problem $P$.

To each M-problem there corresponds a certain degree of difficulty (an
exact definition of degrees of difficulty is given below). It is
possible to define in a natural fashion conjunction, disjunction, and
other operations of propositional calculus on the degrees of
difficulty. As was established by Yu.\ T. Medvedev, the calculus of
M-problems is an interpretation of constructive propositional
calculus. This is to be expected, since the calculus of M-problems is
an elaboration of A. N. Kolmogorov's calculus of problems (see
\cite{Kolmogoroff32}). The definition of reducibility of a family of
functions (M-problems), which is basic in the calculus of M-problems,
has a constructive character.

\begin{dfnA}
  An M-problem $P(A)$ (family $A$) is \emph{reducible} to an M-problem
  $P(B)$ (family $B$), if there exists a general method of
  transformation of any solution of the M-problem $P(B)$ into a
  solution of the M-problem $P(A)$, or more accurately, if there
  exists a partial recursive operator $T$, which transforms each
  function $f$ from the family $B$ into some function $g$ (which
  depends on $f$) from the family $A$, $g=T[f]$. The reducibility of
  the family $A$ (M-problem $P(A)$ to $P(B)$) to $B$ is denoted by
  $A\le B$ ($P(A)\le P(B)$). The family of functions $A$ (the
  M-problem $P(A)$)\footnote{The definitions presented here apply
    equally well to families of functions and to the M-problems which
    they define.} is called \emph{solvable}, if it contains at least
  one general recursive function.  The M-problems (families) $A$ and
  $B$ are called \emph{equivalent} ($A\approx B$) if they are
  reducible to each other.
\end{dfnA}

The class of M-problems equivalent to an M-problem $A$ is called the
\emph{degree of difficulty} of the M-problem $A$ and is denoted by
$a=|A|$. The degrees of difficulty form a partially ordered set
$\Omega$: $|A|=a\le b=|B|$ if the M-problem $A$ is reducible to
$B$. $\Omega$ is a distributive lattice with implication and has a
largest and a smallest element (see \cite{MedvedevA}). The
investigation of M-problems, initiated by Yu.\ T. Medvedev, was
continued by the author in \cite{Muchnik56}.

\subsection*{Section 1}
\addcontentsline{toc}{subsection}{Section 1}

\setcounter{equation}{0}

1. We describe here a second approach to the concept of reducibility
of algorithmic problems, corresponding to classical, i.e.,
non-constructive, formulations.

Along with the problem of constructing an algorithm which solves a
certain problem, it is possible to consider the problem of the
existence of a required algorithm, without insisting on its concrete
form. Then each condition imposed on the arithmetical functions (i.e.,
each family of functions) will be linked to two problems:
\begin{enumerate}
\item The problem of constructing one of the functions of this family:
  the M-problem.
\item The problem of proving the existence of a general recursive
  function in this family.\footnote{It is easy to see here an analogy
    with the question of the existence of a solution of a differential
    equation and the problem of effectively finding a solution.}
\end{enumerate}

Problems of the second type will be called \emph{Ex-problems}. There
is a pairwise one-to-one correspondence between the classes of
families of functions, M-problems, and Ex-problems. The Ex-problem
corresponding to the family of functions $A$ (M-problem $P(A)$) will
be denoted by $\Pbar(A)=Q(A)$. The functions of the family defining
the Ex-problem $Q$ will accordingly be called the \emph{solution
  functions} of the Ex-problem $\Qbar$. An Ex-problem is called
\emph{solvable} if its solution functions include a general recursive
one.

An important method of establishing the solvability of an algorithmic
problem $A$ is to reduce this problem to a different problem $B$, the
solvability of which has already been established. Conversely, the
unsolvability of a problem $A$ implies the unsolvability of any
problem $B$ to which problem $A$ is reducible.

\begin{dfnA}
  The Ex-problem $\Pbar(A)$ (family of functions $A$) is
  \emph{weakly reducible} to the Ex-problem $\Pbar(B)$ (family
  $B$) ($A\vartriangleleft B$), if for any function $f$ of the family
  $B$ ($f\in B$) there exists a partial recursive operator $T$, which
  transforms the function $f$ into the function $g$ of the family $A$
  ($g\in A$).
\end{dfnA}

The choice of the function $f$ governs here not only $g$ but also the
operator $T=T_f$. In this case we say that the problem
$\Pbar(A)$ (family $A$) \emph{reduces weakly} to the problem
$\Pbar(B)$ (family $B$) by means of the operators $\{T_f\}$.

The reducibility of families of functions (problems) in the sense of
Medvedev's definition will be called henceforth \emph{strong
  reducibility} (or simply \emph{reducibility}).

Inasmuch as each of the three objects: the family of functions, the
M-problem, and the Ex-problem, defines uniquely the two others, we
shall henceforth identify these objects and call them \emph{problems}.

\medskip

2. A natural question arises concerning the relation between these
types of reducibility. It is clear that strong reducibility of a
problem $A$ to a problem $B$ implies weak reducibility of $A$ to
$B$. As shown by the example considered below, the converse is
generally not true.

Let the problem $A$ be determined by a family consisting of one
non-recursive function $f$, $A=K_f=\{f\}$, and let problem $B$ be
determined by a family consisting of all the functions obtained from
$f$ in the following manner: for each tuple of natural numbers
$\nbar=\{n_1,\ldots,n_s\}$ we consider the function
$f_{\nbar}(m)$:
\[
f_{\nbar}(m)=\left\{\begin{array}{ll}
    n_{i+1}&\hbox{for }0\le i<s,\\
    f(i-s)&\hbox{for }i\ge s,
  \end{array}\right.
\]
i.e., we ``place in front'' of the sequence of values $\{f(i)\}$ the
tuple $\nbar$:
\[
B=K_f'=\{f_{\nbar}\}.
\]
It is easy to see that the problem $K_f$ reduces weakly to the problem
$K_f'$:
\[
K_f\vartriangleleft K_f'.
\]
For any function $f_{\nbar}\in B$ there exists a partial
recursive operator (p.r.o.)\ $T$ which transforms $f_{\nbar}$
into $f$ (by ``discarding'' the first $s$ values of
$f_{\nbar}$, where $\nbar=\{n_1,n_2,\ldots,n_s\}$).

However, the problem $A$ does not strongly reduce to the problem $B$.

Let us assume the opposite, i.e., that there exists a p.r.o.\ $T$
which transforms any function $f_{\nbar}$ into $f$. Any p.r.o.\
$T$ can be specified by means of a recursive sequence of pairs of
tuples (see \cite{Muchnik56,Kuznecov})
\[
\{(d_w,d_w')\},\; w=0,1,2,\ldots.
\]
If the sequence of several first values of the function $h$ forms a
tuple $d$, then we call $d$ a tuple of the function $h$.  We shall
also say that the function $h$ begins with the tuple $d$.  If $h=T[e]$
and $d_w$ is a tuple of the function $e$, then $d_w'$ is a tuple
of the function $h$.  Inasmuch as $T[f_{d_w}]=f$ and $d_w$ is a tuple
of the function $f_{d_w}$, then $d_w'$ is a tuple of the
function $f$ (for each $w$). In view of the fact that $\{d_w'\}$
is a recursive sequence of tuples, the length of which is unlimited
(in the aggregate), the function $f$ is recursive, yet we have assumed
it to be non-recursive. This contradiction proves that the problem $A$
does not reduce strongly to $B$.

In the foregoing example, the problem $B$ was chosen somewhat
artificially. For algorithmic problems which are usually considered in
the theory of algorithms and its applications, the situation is
different. If we confine ourselves to reducibility (strong and weak)
by means of general recursive operators\footnote{A general
  recursive operator is a p.r.o.\ which transforms functions which are
  everywhere defined (on $\NN$) into functions which are everywhere
  defined.} or even partial recursive operators applicable to each
solution function of the problem to which we reduce another problem,
then both types of reducibility are equivalent for a broad class of
problems.  We shall return to this question in Section 2, and consider
here in greater detail the calculus that results from the definition
of weak reducibility of problems.

\medskip

3. If problems $A$ and $B$ reduce weakly to each other, we shall call
them \emph{weakly equivalent}: $A\wequiv B$. This relation is
transitive, symmetrical, and reflexive. The class of all problems
therefore breaks up into classes of weakly equivalent problems.  The
class of problems which are weakly equivalent to $A$ will be called
the \emph{weak degree of difficulty} of problem $A$.  The weak degree
of difficulty of the problem $A$ characterizes the problem of proving
the existence (in the classical sense) of a computable solution
function of the problem $A$.

A weak degree of difficulty $\bbar$ exceeds $\abar$,
$\bbar\ge\abar$ or $\abar\le\bbar$, if the
problem $A$ reduces weakly to the problem $B$
($\abar=|\Abar|$, $\bbar=|\Bbar|$).  We
denote by $\Omegabar$ the partially ordered set of weak
degrees of difficulty.

Between $\Omega$ and $\Omegabar$ there is a one-sidedly univalent
correspondence $\Omega\to\Omegabar$; to each degree of difficulty
$a\in\Omega$ there corresponds a weak degree of difficulty $\abar$:
$\abar$ is the weak degree of difficulty of a problem $A$ with degree
of difficulty $a$. The correspondence $a\to\abar$ does not depend on
the choice of the problem $A$, since equivalence of problems implies
weak equivalence of problems. This relation is isotopic, since
reducibility of problems implies weak reducibility. The solvable
(smallest) degree $0$ from $\Omega$ corresponds to the solvable weak
degree $\overline{0}$ from $\Omegabar$, and the improper (largest,
i.e., defined by the empty class of functions) degree $\infty$ from
$\Omega$ corresponds to an equal degree from $\Omegabar$. We shall
prove that $\Omegabar$ is a lattice and the indicated correspondence
is a lattice homomorphism.

We note that $\Omegabar$ admits a natural topological
interpretation.  Define a \emph{complete} family of functions or
points of Baire space (complete problem) to be any family (problem)
$A$ having the following property: together with each function $f$
belonging to $A$, the family $A$ contains any function $g$ with
respect to which the function $f$ is recursive.

We shall establish some properties of complete families. The union and
intersection of any number of complete families (finite or infinite)
are also complete families.

Let $B$ be some family of functions. The family consisting of all
functions $\{g\}$, for each of which there exists a certain function
$f$ from $B$, which is recursive with respect to this function $g$
will be called the \emph{completion} $B'$ of the family $B$.  It
is obvious that $B'$ is the smallest complete family containing
the family $B$, and the completion of a complete family $A$ coincides
with $A$: $A'=A$. The family $B'$ is weakly equivalent to
$B$. It is sufficient to establish that $B'\trianglelefteq B$,
since $B'\supset B$, from which follows $B\trianglelefteq
B'$. Indeed, for any function $g\in B'$ there exists a
p.r.o.\ $T$ such that $T[g]=f\in B$.

Completions of two weakly equivalent families $A$ and $B$ coincide:
$A\wequiv B\,\rightarrow\,A'=B'$. Let $g\in
A'$. Then there exists a function $f\in A$ and a p.r.o.\ $T_1$
such that $T_1[g]=f$. In view of $A\wequiv B$, there exists a p.r.o.\
$T$ such that $T[f]=h\in B$. Then $T_2[g]=T[T_1[g]]=h\in B$.
Therefore $g\in B'$. Conversely, if $g\in B'$, then $g\in
A'$. Thus $A'=B'$.

In view of the foregoing, any weak degree $\abar$ defines
uniquely a complete family (problem) $A$, which we shall call the
\emph{representative} of $\abar$.

\begin{lem*A}
  Let $\abar$ and $\bbar$ be weak degrees of complete
  families $A$ and $B$ respectively. Then
  $\abar\ge\bbar\leftrightarrow A\subset B$
  \footnote{$\frakA\leftrightarrow\frakB$ denotes that the statement
    $\frakA$ is equivalent to statement $\frakB$.}, or using a
  different notation
  \begin{equation}
    \Abar\trianglerighteq\Bbar\leftrightarrow A\subset B.
  \end{equation}
\end{lem*A}

Let $g\in A$. Then there exists a p.r.o.\ $T$ such that $T[g]=f\in
B$. In view of the completeness of the family $B$, $g\in B$. The
relation $A\subset B\rightarrow A\trianglerighteq B$ is obvious.

We now readily prove some theorems concerning the properties of
$\Omegabar$.

\begin{thmA}\label{thm:-1}
  For any set of weak degrees $\{\abar_\xi\}$ there exist exact
  upper and lower bounds, denoted by $\bigvee\abar_\xi$ and
  $\bigwedge\abar_\xi$, respectively.
\end{thmA}

Let $A_\xi$ be a complete family with weak degree of difficulty
$\abar_\xi$ and $A=\displaystyle\bigcup_\xi\Abar_\xi$,
$\abar=|\Abar|$. We shall prove that
$\abar=\inf\{\abar_\xi\}$.  Obviously $A$ is a complete
family and $\abar\le\abar_\xi$ for any $\xi$.  Further,
let $\bbar\le\abar_\xi$ for any $\xi$ and $B$ a complete
family, $\bbar=|\Bbar|$.  Then $B\supset A_\xi$ and
$B\supset A$, hence $\bbar\le\abar$. We put
$A^*=\displaystyle\bigcap_\xi A_\xi$. Obviously,
$\abar^*\ge\abar_\xi$ for any $\xi$. In addition, if
$\bbar\ge\abar_\xi$ for all $\xi$, and $B$ is a
representative of $\bbar$, then $B\subset A_\xi$ and
$B\subset\displaystyle\bigcap_\xi A_\xi=A^*$, i.e.,
$\bbar\ge\abar^*$. Hence
$\abar^*=\sup\{\abar_\xi\}$.

From the proof of Theorem \ref{thm:-1}, we see that the operations of
taking the exact upper and lower bounds in $\Omegabar$
correspond to the operations of intersection and union of complete
families of functions.

$\Omegabar$ is a complete lattice, represented by subsets of
the Baire space $J$. In the function space $J$ it is possible to
introduce a topology by assigning as open sets the complete families
of functions. This will be a $T_0$ space. (On this subject see
G. D. Birkhoff, Lattice theory, Russian translation of the 2nd
edition, IL 1952, Chapter IV, \S\S\ 1 and 2.)

\begin{thmA}\label{thm:-2}
  Let $A$ and $B$ be arbitrary problems, $a=|A|,
  \abar=|\Bbar|, b=|B|,
  \bbar=|\Bbar|$. Then the problem $C=A\cup B$
  with degree of difficulty $c=a\lor b$ has a weak degree
  $\cbar=\abar\lor\bbar$, and the problem $D$
  with degree of difficulty $d=a\land b$ has a weak degree
  $\dbar=\abar\land\bbar$.
\end{thmA}

Following Yu.\ T. Medvedev \cite{MedvedevA}, we choose problems $C$
and $D$ in the following fashion. We define the p.r.o.s $R_0$, $R_1$
and a two-place p.r.o.\ $R$:
\[
\begin{array}{rcl}
  f_i(n)&=&R_i[f],\\[6pt]
  f_i(n)&=&\left\{\begin{array}{ll}
      i&\hbox{for }n=0\\
      f(n-1)&\hbox{for }n>0
    \end{array}\right\} (i=0,1),\\[12pt]
  h(n)&=&R[f(m),g(m)],\\[6pt]
  h(n)&=&\left\{\begin{array}{ll}
      f(m)&\hbox{for }n=2m,\\
      g(m)&\hbox{for }n=2m+1.
    \end{array}\right.
\end{array}
\]
The problem $C$ consists of all the functions $f_0(n)=R_0[f]$, where
$f\in A$, and all the functions $g_1(n)=R_1[g]$, where $g\in B$, and
$|C|=c=a\lor b$.
The problem $C_1$ consists of all the solution functions of problems
$A$ and $B$,
\[
|\cbar_1|=\cbar=\abar\lor\bbar.
\]
We shall prove that $C$ and $C_1$ are weakly equivalent, i.e.,
$|\Cbar_1|=|\Cbar|$. Indeed, each
function $h\in C_1$ can be transformed with the aid of $R_0$ or $R_1$
into a function $h_i\in C$, and each function $h_i\in C$ can be
reduced by means of an inverse transformation into $h\in C_1$ (i.e.,
$C_1$ reduces even strongly to $C$).

Further, the problem $D$ consists of all the functions $h=R[f,g]$,
where $f$ runs through class $A$ and $g$ through class $B$.  The
problem $D_1$ consists of all the functions $e$ such that problems $A$
and $B$ reduce to the problem of computability $A_e=\{e\}$, i.e., for
each function $e$ there exists p.r.o.\ $T_1$ and $T_2$ such that
$T_1[e]\in A, T_2[e]\in B$. We shall prove that problems $D$ and $D_1$
are weakly equivalent:
\begin{enumerate}
\item $D_1\trianglelefteq D$ (even $D_1\le D$). The relation $D_1\le
  D$ follows from the fact that class $D$ is contained in $D_1$, since
  any function $h\in D$ can be transformed with the aid of the p.r.o.\
  $T_1$ ($T_2$) into the function $f(g)$, $f\in A$ ($g\in B$). To this
  end it is sufficient to put
  \[
  \begin{array}{l}
    T_1[h]=f(m)=h(2m),\\
    T_2[h]=g(m)=h(2m+1).
  \end{array}
  \]
\item $D\trianglelefteq D_1$. Let the function $e\in D_1$. Then there
  exist p.r.o.\ $T_1$ and $T_2$ such that
  \[
  f=T_1[e]\in A,\,g=T_2[e]\in B,\,R[f,g]=h\in D
  \]
  and
  \[
  h=T[e]=R\left[T_1[e],T_2[e]\right]\in D.
  \]
  The p.r.o.\ $T$ transforms the function $e$ into $h\in D$, from
  which it follows that $D\trianglelefteq D_1$.
\end{enumerate}

The weak degree $\cbar=\abar\lor\bbar$ will be called the
\emph{disjunction}, and $\dbar=\abar\land\bbar$ will be called the
\emph{conjunction}, of the weak degrees $\abar$ and $\bbar$. Let us
prove that the lattice $\Omegabar$ has an implication operator:

\begin{thmA}\label{thm:-3}
  For any weak degrees $\abar$ and $\bbar$ there exists a smallest
  degree $\cbar^*$ in the class of weak degrees $\cbar$ such that
  $\abar\land\cbar\ge\bbar$.
\end{thmA}

\emph{Proof.}  We consider the representatives of the weak degrees
$\abar$ and $\bbar$, i.e., the complete families (problems) $A$ and
$B$, $|\Abar|=\abar$, $|\Bbar|=\bbar$.  We denote by $C^*$ the family
of all the functions $\{g\}$ such that for each pair of functions
$[f,g]$, where $f\in A$ and $g\in C^*$, there exist a p.r.o.\ $T$
which transforms the pair $[f,g]$ into a function $e\in B$,
$e=T[f,g]$. It is obvious that the family $C^*$ includes the family
$B$ and that $\abar\land\cbar^*\ge\bbar$ where
$\cbar^*=|\Cbar^*|$. Let us prove that the problem $C^*$ reduces
weakly to any problem $C$ such that $\abar\land\cbar\ge\bbar$ where
$\cbar=|\Cbar|$. Let $C$ be such a problem and $g$ an arbitrary
function from $C$. As follows from Theorem \ref{thm:-2}, the problem
$D$, which consists of all of the functions $h=R[f,g]$ where $f$ runs
through the family $A$ and $g$ through the family $C$, has the weak
degree $\dbar=\abar\land\cbar$.  Inasmuch as $\dbar\ge\bbar$, i.e.,
$|D|\trianglelefteq|B|$, for any function $h\in D$, there exists a
p.r.o.\ $T_1$ such that $e=T_1[h]\in B$. This means that for any pair
of functions $[f,g]$ where $f\in A$ and $g\in C$, there exists a
two-place p.r.o.\ $T=T_1R$ such that $e=T[f,g]=T_1[R[f,g]]\in B$. By
definition of $C^*$, the function $g\in C^*$. It follows therefore
that $C\subset C^*$ and $|C^*|\trianglelefteq|C|$.  This completes the
proof.

We shall call $C^*$ the \emph{weak problem of reducibility} of the
problem $B$ to the problem $A$, and $\cbar^*$ the implication, denoted
by $\abar\supset\bbar$. Obviously $C^*$ is a complete family, i.e.,
the representative of $\cbar^*$.

We note that implication, generally speaking, is not conserved in
homomorphism of the lattices $\Omega\to\Omegabar$.  Indeed, in the
example discussed in Section 2, the problems $A$ and $B$ ($A=K_f$,
$B=K_f'$) were related by $|A|>|B|$ and $|\Abar|\wequiv|\Bbar|$, or
$a>b$ and $\abar=\bbar$.  Therefore the implication
$\bbar\supset\abar$ is the solvable (trivial) weak degree (i.e., the
degree of a solvable problem), and $b\supset a$ is an unsolvable
degree and $\overline{b\supset a}\ne\bbar\supset\abar$.

Note that solvability of the weak degree $\abar\supset\bbar$ is
equivalent to the relation $\abar\ge\bbar$.  The proof of this is
simple and will be omitted.  Further consideration of this point is
analogous to that of Yu.\ T. Medvedev with respect to the calculus of
$\Omega$.

We consider an arbitrary segment $\Omegabar:0\le x\le\dbar$.  The weak
degree $\lnot x=x\supset d$ is called the \emph{negation} of the weak
degree $x$ (with respect to $\dbar$).  We introduce also the notation
$\abar\sim\bbar$ for the degree
$(\abar\supset\bbar)\cap(\bbar\supset\abar)$.

The thought arises of the connection between the calculus of weak
degrees $\Omegabar$ and the propositional calculus: elementary
propositions can be interpreted as weak degrees, and the operations of
propositional calculus correspond to like operations of the calculus
of weak degrees. The truth of a formula corresponds to the solvability
of a weak degree.

\begin{thmA}\label{thm:-4}
  All the axioms and rules of derivation of intuitionistic
  propositional calculus are satisfied for weak degrees of an
  arbitrary segment $0\le x\le\dbar$ in $\Omegabar$.
\end{thmA}

Theorem \ref{thm:-4} follows from the existence of implication in the
distributive lattice $\Omega$ (see Birkhoff, Lattice theory, Russian
translation of the 2nd edition, Chapter XII, \S 7).

Let us discuss the consequences of this point. In spite of the fact
that the definition of weak reducibility of problems has been chosen
in accordance with classical premises, the calculus of weak degrees
obtained thereby is an interpretation of constructive propositional
calculus and does not include, for example, the law of the excluded
third.

However, this should not surprise us, since the calculus of weak
degrees $\Omegabar$, like that of $\Omega$, is a refinement of
Kolmogorov's calculus of problems.

The question whether the weak degrees $\Omegabar$ are an
\emph{exact}\footnote{An interpretation of a logical calculus $K$ is
  called exact if all formulas true (solvable, realizable) in the
  interpretation are derivable in the calculus $K$.} interpretation of
constructive propositional calculus remains open.  We note that the
calculus of degrees of difficulty $\Omega$, as shown recently by Yu.\
T. Medvedev, is an exact interpretation of the constructive calculus.

\subsection*{Section 2}
\addcontentsline{toc}{subsection}{Section 2}

\setcounter{equation}{0} 

In this section we analyze the question of the relation between strong
and weak reducibility under certain limitations on the p.r.o.s by
means of which the reducibility is realized, and on the problems
themselves.

We need several new concepts. In the arguments that follow we shall
find it convenient to use the Baire space $J$.

Arithmetical functions can be interpreted as points in Baire space
(considering the sequence of the values of these functions
\cite{Aleksandrov,Kuznecov}). To each problem $A$ in such an
interpretation, there corresponds a certain set of points $\frakM_A$
of the Baire space, which defines it completely.

Let $\delta_{\nbar}$ be a Baire interval, defined by a tuple
$\nbar=(n_1,n_2,\ldots,n_s)$. The problem which is defined by the set
of points $\frakM_A\cap\delta_{\nbar}$ shall be called the
\emph{interval} $A_{\nbar}$ of the problem $A$. In other words,
$A_{\nbar}$ is defined by the class of solution functions of the
problem $A$ beginning with the tuple $\nbar$. The interval $A_{\nbar}$
is called non-empty if the set $\frakM_A\cap\delta_{\nbar}$ is
non-empty.  A problem $A$ is called \emph{uniform} if any of its
non-empty intervals is (strongly) reducible to it.

The \emph{problem of solvability} $A_E$ of the set $E$ is defined by
the class $K_A(E)$, consisting of one characteristic function of the
set $E$. The \emph{problem of enumerability} $C(E)$ is determined by
the class $K_C(E)=\{f(n)\}$ of the functions that enumerate the set
$E$, i.e., the set $E$ is the image of the function $f(n)$.  The
\emph{problem of separability} $A_{E_0E_1}$ of the sets $E_0$ and
$E_1$ with empty intersection is determined by the class of functions
$\{f(n)\}$ satisfying the condition
\begin{equation}\label{eqn:-2.1}
  f(k)=\left\{\begin{array}{cl}
      0&\hbox{for }k\in E_0,\\
      1&\hbox{for }k\in E_1,\\
      0\hbox{ or }1&\hbox{for }k\notin E_0\cup E_1
    \end{array}\right.
\end{equation}

\begin{thmA}\label{thm:-5}
  The problem of enumerability of any non-empty set $E$ is uniform.
\end{thmA}

Indeed, let $C$ be the problem of enumerability of the set $E$ and let
$C_{\nbar}$ be a non-empty interval in it:
$\nbar=(n_1,n_2,\ldots,n_s)$.  Thus $K_{C_{\nbar}}$ consists of all
the functions which enumerate the set $E$ and begin with the tuple
$\nbar$.  The problem $C_{\nbar}$ is (strongly) reducible to the
problem $C$ by means of the p.r.o.\ $T$ which, being applied to any
function $f$, shifts the sequence of its values by first adding the
tuple $\nbar$.

\begin{thmA}\label{thm:-6}
  The problem of separability $A_{E_0E_1}$ is uniform for arbitrary
  $E_0, E_1$ ($E_0\cap E_1=\Lambda$, where $\Lambda$ is the empty
  set).
\end{thmA}

We denote the problem $A_{E_0E_1}$ by $A$. Let $A_{\nbar}$ be a
non-empty interval of the problem $A$, $\nbar=(n_1,n_2,\ldots,n_s)$.
It is obvious that
\begin{equation}\label{eqn:-2.2}
  n_k=\left\{\begin{array}{cl}
      0&\hbox{for }k\in E_0,\\
      1&\hbox{for }k\in E_1,\\
      0\hbox{ or }1&\hbox{for }k\notin E_0\cup E_1
    \end{array}\right\}, (k=1,2,\ldots,s)
\end{equation}

The problem $A_{\nbar}$ is (strongly) reducible to the problem $A$ by
means of the p.r.o.\ $T$ which replaces the first $s$ values of any
function by the tuple $\nbar$.  If $f(k)$ is a solution function of
the problem $A$, then it satisfies the condition (\ref{eqn:-2.1}). But
then the function $g=T[f]$ also satisfies the condition
(\ref{eqn:-2.1}), as follows from (\ref{eqn:-2.2}) and from the
definition of the p.r.o.\ $T$. In addition, the function $g$ begins
with the tuple $\nbar$ and hence is a solution function of the
M-problem $A_{\nbar}$, which was to be proved.

The \emph{problem of continuation} of the partial\footnote{I.e.,
  perhaps not everywhere defined.} function $f(m)$ is the problem
$B_f$ defined by the class of functions (which are defined everywhere
on $\NN$) coinciding with the function $f(m)$ wherever the latter is
defined. (We shall call such functions \emph{continuations} of
$f(m)$.) We note that a problem of separability is a particular case
of a problem of continuation.  Obviously we have:
\begin{thmA}\label{thm:-7}
  The problem of continuation of any partial function is uniform.
\end{thmA}

The proof of Theorem \ref{thm:-7} is analogous to the proof of Theorem
\ref{thm:-6}.

Inasmuch as the p.r.o.s used in the proofs of Theorems \ref{thm:-1}
and \ref{thm:-2} are general recursive, each problem of enumerability
or separability reduces to any of its non-empty intervals by means of
a general recursive operator.  Problems possessing this property will
be called \emph{general recursively uniform}.  In addition to problems
of enumerability and separability, problems of solvability are also
general recursively uniform, since the operator of identical
transformation reduces any function to itself.

An example of a non-uniform problem is the problem defined by the
class $K=\{f,g\}$, where the degree of non-computability of the
function $f$ is strictly greater than the degree of non-computability
of the function $g$.

We shall call the problem $B$ \emph{closed} if it corresponds to a
closed set of points $\frakM_B$ of the Baire space $J$.  Obviously,
solvability problems are closed.

\begin{thmA}\label{thm:-8}
  The continuation problem of any partial function is closed.
\end{thmA}

Let $\{g_k(m)\}$ be a convergent sequence of continuations of the
function $f(m)$, and let $g(m)=\displaystyle\lim_{k\to\infty}g_k(m)$.
We shall prove that $g(m)$ also continues $f(m)$. If the function
$f(m)$ is defined for $m=m_0$, then $g_k(m_0)=f(m_0)$ for all $k$.
Consequently $g(m_0)=f(m_0)$, as was to be proved.

\begin{cor*A}
  Any problem of separability is closed.
\end{cor*A}

\begin{thmA}\label{thm:-9}
  The problem of enumerability $C(E)$ of any set $E$ containing more
  than one element is not closed.
\end{thmA}

Let $a\in E$ and let the function $f(m)$ enumerate the set $E$. We
define the sequence of functions $\{f_k(m)\}$ enumerating the set $E$:
\[
f_k(m)=\left\{\begin{array}{cl}
    a&\hbox{for }m<k,\\
    f(m-k)&\hbox{for }m\ge k.
  \end{array}\right.
\]
Obviously
\[
\lim_{k\to\infty}f_k(m)=g(m)\equiv a.
\]
In view of the fact that the set $E\setminus\{a\}$ is not empty,
$g(m)$ is not a solution function of the problem $C(E)$, and
consequently the problem $C(E)$ is not closed.

However, it is possible to generalize the concept of closedness of a
problem in such a way that enumerability problems as well as many
other problems which are of interest for the recursive theory of sets
are included.  This concept is closely related with the theory of
infinite games \cite{Stockii}.

Let $S$ be a set of points of the Baire space $J$. We imagine two
players I and II who move alternately, and their moves consist of
choosing Baire intervals which intersect with the set $S$. Player I
chooses as his first move the Baire interval $\delta_1$, which
intersects with $S$. If player I chose in move $m$ the Baire interval
$\delta_m$, then player II chooses in the $m$th move a
sub-interval\footnote{We consider Baire sub-intervals which are proper
  parts of their intervals.} $\delta_m^*$ of the interval $\delta_m$
intersecting with $S$.  Then player I chooses in the $(m+1)$st move a
sub-interval $\delta_{m+1}$ of the interval $\delta_m^*$ intersecting
with $S$. Let us assume that after the $m$th move of player I (II) the
play is in the interval $\delta_m$ ($\delta_m^*$). We agree that at
the beginning the game is in the interval $\delta_0=J$. Player II
\emph{wins} if the sequence of intervals
\[
\delta_0\supset\delta_1\supset\delta_1^*\supset
\delta_2\supset\delta_2^*\supset\ldots\supset
\delta_m\supset\delta_m^*\supset\delta_{m+1}\ldots
\]
contracts to a point of the set $S$. Otherwise, player I wins.

We fix once and for all some effective numbering of the Baire
intervals by means of natural numbers. By a \emph{strategy} of a
player we mean a function $r=\varphi(n)$ which indicates for each
interval with number $n$ the number $r$ of a sub-interval of it. If the
game is in the interval numbered $n$, then the player making the next
move chooses the interval with number $\varphi(n)$. The strategy
$\varphi$ is called \emph{correct} with respect to the set $S$ if for
any interval numbered $n$ intersecting with $S$, the interval numbered
$\varphi(n)$ also intersects with $S$.  We shall henceforth take
strategy to mean a strategy which is correct with respect to the
considered set. A strategy is called \emph{winning} (for the set $S$)
if player II, using this strategy, wins for any correct strategy of
his opponent.  A set $S$ is called \emph{winning} if there exists a
winning strategy for this set. The M-problem $A(S)$ and the class of
functions $K(S)$ defined by the set $S$ will also be called winning in
this case.

A set $S$ (a problem $A(S)$) the complement of which is nowhere dense
is called \emph{trivially winning}. In order to win, it is sufficient
for player II to choose as his first move an interval which is
completely contained in $S$, which is possible since the complement
$CS$ is nowhere dense.

If neither the set $S$ nor its complement $CS$ is trivially winning
(or equivalently, neither $S$ nor $CS$ is nowhere dense), then they
cannot be simultaneously winning. In fact, let $S$ be a winning set
and $\varphi(n)$ its winning strategy.  Let us consider the game with
respect to $CS$. Player I chooses as his first move an interval in
which the set $S$ is everywhere dense (such an interval exists, since
$S$ is not a set which is nowhere dense). Then player I applies
strategy $\varphi(n)$. For any strategy of player II, the sequence of
intervals in which the game is situated will contract to a point
belonging to $S$, i.e., player II loses.

An example of a winning set (problem) is a closed set $S$ (problem
$A(S)$). In this case the sequence of intervals
$\{\delta_m,\delta_m^*\}$ contracts always to a point of $S$.  It
follows therefore that problems of solvability and separability are
winning. There exist also non-closed winning problems.
\begin{thmA}\label{thm:-10}
  Any problem of enumerability is a winning problem.
\end{thmA}

Let $G_F$ be the problem of enumerability of a set $F$ of natural
numbers, and let $\frakM$ be the corresponding subset of $J$.  Player
II chooses a strategy $r=\varphi(n)$ in the following manner: let $n$
be the number of an interval $\delta_n=(n_1,n_2,\ldots,n_s)$
containing points in $\frakM$, by virtue of which
$n_1,n_2,\ldots,n_l\in F$. We denote by $n_{l+1}$ the smallest number
belonging to the set $F$ which is not equal to $n_i$ for
$i=1,2,\ldots,l$, and if there is no such number, then
$n_{l+1}=n_l$. We put
$\delta_{\varphi(n)}=(n_1,n_2,\ldots,n_l,n_{l+1})$. Obviously
$\delta_{\varphi(n)}$ intersects with $\frakM$.  We consider the
sequence of intervals in which the game occurs:
\[
\delta_1,\delta_1^*,\ldots,\delta_m,\delta_m^*,\ldots.
\]
We note that: (1) for each $m$, all of the numbers of the
tuples\footnote{As is well known, Baire intervals are identified with
  tuples of natural numbers.}
\begin{center}
  $\delta_m=(n_1,n_2,\ldots, n_p)$ and
  $\delta_m^*=(n_1,n_2,\ldots,n_{p^*})$
\end{center}
belong to $F$; (2) any number $q\in F$ will be sooner or later
encountered in the tuples $\{\delta_m,\delta_m^*\}$, because going
from $\delta_m$ to $\delta_m^*$ we add a still unchosen element of the
set $F$ (if it exists). But then the sequence
$\{\delta_m,\delta_m^*\}$ contracts to the point
$b=(n_1,n_2,n_3,\ldots,n_l,\ldots)$ where the set $\{n_i\}$ coincides
with $F$, i.e., $b\in\frakM$. This proves the theorem.

A partial recursive operator $T$ is called \emph{fully applicable} to
a problem $B$ if it is defined on each solution function of the
problem $B$. The class of all p.r.o.s which are fully applicable to
the problem $B$ will be denoted by $Q_B$.

If a problem $A$ reduces strongly (weakly) to a problem $B$ by means
of operators of a certain class $P$, then we say that $A$ is strongly
(weakly) $P$-reducible to $B$.

Let $U$ be some class of p.r.o.s.  A problem $B$ is called
\emph{$U$-uniform} if any of its non-empty intervals is strongly
reducible to it by means of operators of the class $U$.

The class of p.r.o.s represented in the form of
compositions\footnote{$RT[f]$ is the result of successive application
  of the p.r.o.s $R$ and $T$ to the function $f$.} $RT[f]$ where $R\in
P$ and $T\in U$ will be denoted by $PU$.

\begin{thmA}\label{thm:-11}
  Let $A$ be a closed problem, $B$ a $U$-uniform winning problem, and
  $P$ a subclass of $Q_B$. If the problem $A$ is weakly $P$-reducible
  to the problem $B$, then the problem $A$ is strongly $PU$-reducible
  to the problem $B$.
\end{thmA}

\textit{Proof.}  Let us assume that no operator of class $PU$ reduces
the problem $A$ to the problem $B$. We arrange all the p.r.\ operators
of the class $P$ in some sequence
\[
T_1, T_2, T_3, \ldots, T_s,\ldots.
\]  
To each operator $T_s$ there corresponds a continuous function
$\theta_s$ in the Baire space (see \cite{Kuznecov}) defined at each
point of the set $\frakM_B$, by virtue of $P\subset Q_B$.  By virtue
of our assumptions, including continuity of the functions $\theta_s$
and closedness of the set $\frakM_A$, there exists for each $s$ an
interval $\delta$ represented by the function $\theta_s$ in
$C\frakM_A$, the complement of $\frakM_A$. Let $\varphi$ be the
winning strategy for the set $\frakM_B$.  By the method indicated
above, we obtain for $T_1$ an interval $\delta=\delta_1$. Let $n_1$ be
the number of $\delta_1$; let $r_1=\varphi(n_1)$; let $\delta_1^*$ be
the interval numbered $r_1$; let $B_1$ be the problem defined by the
set $\frakM_B\cap\delta_1^*$, which is then a non-empty interval of
the problem $B$.

Inasmuch as the problem $B$ is $U$-uniform, the problem $B_1$ is
strongly $U$-reducible to $B$. But then the problem $A$ cannot be
strongly $P$-reducible to $B_1$, since in accordance with our
assumption the problem $A$ is not strongly $PU$-reducible to
$B$. Consequently, there exists a non-empty interval $\delta_2$
intersecting with the set $\frakM_B$ and transformed by the function
$\theta_2$ into a subset of $C\frakM_A$. Obviously it is possible to
choose $\delta_2$ so as to make $\delta_2\subset\delta_1^*$. If $n_2$
is the number of $\delta_2$, then $r_2=\varphi(n_2)$ is the number of
a sub-interval $\delta_2^*$, $\delta_2^*\subset\delta_2$, which also
intersects with $\frakM_{B_1}$. Let $B_2$ be the problem defined by
the set
\[
\frakM_{B_1}\cap\delta_2^*=\frakM_B\cap\delta_2^*.
\]
We define further in the same manner the intervals
\[
\delta_3\supset\delta_3^*\supset\delta_4\supset\delta_4^*\supset\ldots
\]
and the problems $B_3,B_4,\ldots$.

The problem $B$, by virtue of $U$-uniformity, is strongly
$U$-reducible to any problem $B_s$, while the problem $A$ does not
reduce strongly to $B_s$ by any $P$-operator.  By virtue of the
winning character of the problem $B$ and of the strategy $\varphi$,
the sequence $\{\delta_s,\delta_s^*\}$ contracts to a point
$f\in\frakM_B$.

Inasmuch as for any $s$ the operator $T_s$ transforms the set
$\frakM_{B_s}$ into a subset of $C\frakM_A$, we have $T_s[f]=g\in
C\frakM_A$ for any $s$, and this means that the problem $A$ does not
reduce weakly to $B$ by means of operators in the class $P$, which
contradicts the conditions of the theorem.  Consequently, the
assumption that the problem $A$ does not reduce strongly to $B$ by
means of operators in the class $PU$ is incorrect.  The theorem is
proved.

Our desire to be as general as possible has made it necessary to
formulate the theorem in a rather cumbersome manner.  We present some
simply formulated corollaries of Theorem \ref{thm:-11}.

\begin{corA}\label{cor:11.1}
  If a closed problem $A$ reduces weakly to a uniform winning problem
  $B$ by means of operators of the class $Q_B$, then $A$ reduces
  strongly to $B$.
\end{corA}

A $U$-uniform problem is called \emph{general recursively uniform} if
$U$ is the class of general recursive operators
\cite{Trahtenbrot55,Kuznecov}. Problems of solvability, separability,
and enumerability are general recursively uniform.

\begin{corA}\label{cor:11.2}
  If a closed problem $A$ reduces weakly to a general recursively
  uniform problem $B$ by means of general recursive operators, then
  problem $A$ reduces strongly to problem $B$ by means of a general
  recursive operator.
\end{corA}

In conclusion, we formulate some unsolved problems.

\begin{enumerate}
\item Is it possible to strengthen the fundamental theorem in such a
  way that weak reducibility (by means of arbitrary partial recursive
  operators) of a closed problem $A$ to a uniform problem $B$ would
  imply strong reducibility of $A$ to $B$?
\item Under what ``natural'' conditions imposed on problems $A$ and
  $B$ does weak reducibility (by means of an arbitrary p.r.o.) imply
  strong reducibility?
\item What is the situation in the particular case when $A$ is a
  solvability problem and $B$ is a separability problem of enumerated
  recursively inseparable sets (we note that no non-trivial
  solvability problem $A$ can be reduced strongly to a separability
  problem $B$ \cite{Muchnik56,Muchnik65}).
\end{enumerate}

\noindent Received 5 July 1962

\subsection*{Literature}
\addcontentsline{toc}{subsection}{Literature}
\renewcommand{\refname}{

\end{document}